\documentclass[11pt,a4paper]{article}

\UseRawInputEncoding
\usepackage{titling}
\usepackage{fullpage,authblk}
\usepackage{amssymb,amsfonts,amsmath,enumerate}
\usepackage{graphicx,bm,latexsym}
\usepackage{multirow}
\usepackage{algorithm}
\usepackage{algpseudocode}
\usepackage{mathrsfs}

\usepackage{tocloft}
\usepackage[bookmarks, colorlinks=true, plainpages = false, citecolor = green, urlcolor = black, filecolor = black]{hyperref}
\usepackage{lineno}

\newtheorem{theorem}{Theorem}
\newtheorem{proof}{Proof}
\newtheorem{corollary}{Corollary}
\newtheorem{remark}{Remark}
\newtheorem{proposition}{Proposition}
\newtheorem{definition}{Definition}

\begin{document}
\title{Geometric Considerations of a Good Dictionary \\ for Koopman Analysis of Dynamical Systems{\color{black}: \\Cardinality, ``Primary Eigenfunction," and Efficient Representation}} 
         
\author{Erik M.~Bollt}
\affil{Department of Electrical and Computer Engineering and $C^3S^2$ the Clarkson Center for Complex Systems Science, Clarkson University, Potsdam, New York 13699}

\date{\today}

\maketitle

\begin{abstract}
     Representation of a dynamical system in terms of simplifying modes is a central premise of reduced order modelling and a primary concern of the increasingly popular DMD (dynamic mode decomposition) empirical interpretation of Koopman operator analysis of complex systems.  In the spirit of optimal approximation and reduced order modelling the goal of DMD methods and variants are to describe the dynamical evolution as a linear evolution in an appropriately transformed lower rank space, as best as possible.  
     {\color{black} 
     That Koopman eigenfunctions follow a linear PDE that is solvable by the method of characteristics yields several interesting relationships between geometric and algebraic properties.  
     Corresponding to freedom to arbitrarily define functions on a data surface, for each eigenvalue, there are infinitely many eigenfunctions emanating along characteristics.  We focus on contrasting cardinality and equivalence.  In particular, we introduce an equivalence class, ``primary eigenfunctions," consisting of those eigenfunctions with identical sets of level sets, that helps contrast algebraic multiplicity from other geometric aspects.
     Popularly, Koopman methods and notably dynamic mode decomposition (DMD) and variants, allow data-driven study of how measurable functions evolve along orbits.
As far as we know, there has not been an in depth study regarding the underlying geometry as related to an efficient representation.        We present a construction that leads to functions on the data surface whose corresponding eigenfunctions are efficient in a least squares sense.  We call this construction optimal Koopman eigenfunction DMD, (oKEEDMD), and we highlight with examples.}      
                      \end{abstract}



\section{Introduction}

The data-oriented Koopman eigenfunction analysis perspective for analyzing dynamical systems has become extensively popular and relevant  lately in science and engineering \cite{budivsic2012applied, kutz2016dynamic, lan2013linearization}, as a way to interpret even a nonlinear dynamical system, as a much simpler linear system.  However in turn, a  possibly finite dimensional nonlinear system, may well require infinitely many dimensions to interpret as a linear system.  This is perhaps not a bad trade-off, infinite dimensions for linearity, and from there, computational schemes proceed to estimate the representation in terms of finite dimensional truncation of the infinite dimensional embedding.   In this paper we will introduce a new perspective on the nature of eigenfunctions, relating algebraic and geometric aspects.  The fact that there is a coincidence of these eigenfunctions and duplicity of eigenvalues, and even their form of their explicit solution, was  already observed in  \cite{korda2019optimal} as well as our prior work, \cite{bollt2018matching}; we go a step further in this work by 
describing a quotient over those functions with matching level sets, that we call primary eigenfunctions.   {\color{black} Further, we go on to describe a geometry amongst these primary eigenfunctions in terms of the manner in which their level sets intersect the level sets of other primary eigenfunctions.  Then,} this perspective leads us to  present a new data driven approach toward optimal empirical eigenfunctions by, as it turns out, construction on a lower dimensional data set, transverse to the flow.  As an aside, also we will discuss how the nature of the point and continuous spectra are domain dependent.   

{\color{black}
This paper offers the following points, along with supporting technical materials:
\begin{enumerate}
\item Koopman eigenfunctions are solutions of a linear PDE which is solvable by means of the method of characteristics that propagates an initial data function defined on a co-dimension one manifold that is transverse to the flow of the underlying system.   It follows that the cardinality of eigenfunctions associated with each eigenvalue is the same as the cardinality of data functions that can be stated on the data manifold.
\item While the standard algebraic property of eigenfunctions also describes infinitely many eigenfunctions, this does not account for all of the eigenfunctions.  
We define a new concept, called ``primary eigenfunctions" as an equivalence class of those eigenfunctions which share the same set of level sets. We contrast a geometric interpretation versus  the standard algebraic concept.
\item In the title of this paper, the phrase geometric considerations refers to the fact that primary eigenfunctions are distinguished by the manner in which their  level sets intersect.  However, a good dictionary refers to a set of eigenfunctions for an efficient representation.
\item When representing  arbitrary observable functions by series of eigenfunctions, there is freedom to define an efficient representation, choosing the best amongst the rich and infinite number of eigenfunctions.  
We develop a construction of a good dictionary of eigenfunctions, called  optimal Koopman eigenfunction extended DMD, (oKEEDMD), that efficiently represents a specific chosen observable function, and that explicitly exploits the method of characteristics construction.  
\end{enumerate} 
}

Spectral properties of the Koopman operator, and especially the Koopman mode decomposition (called KMD, \cite{mezictraffic}) offer a data-driven methodology for global analysis of a dynamical systems, including forecasting as well as descriptive decompositon into growing, fluctuation and decaying components, \cite{mezic2004comparison,mezic2005spectral,mezic2019spectrum}.  It is the many data-driven aspects of this concept that have opened a floodgate of research into this spectral analysis philosophy for dynamical systems.
The initial DMD and exact DMD methods working directly from the data are still useful and popular numerical methods,
\cite{schmid-2010,rowley2009spectral, kutz2016dynamic} and variants such as sparse DMD \cite{jovanovic2012low}, extended dynamic mode decomposition (EDMD), that interpet data in terms of basis functions, \cite{williams-2015,williams-2015b,kevrekidis2016kernel}, that may be called a dictionary of features using the language of machine learning.  This includes extended dynamic mode decomposition with (optimal) dictionary learning EDMD-DL, \cite{li2017extended}).  Many other specialized and extremely useful variants exist including such as DMD-c (with control)  \cite{kutz2016dynamic, kaiser2020data} emphasize developing control laws.  Basis learned from data through eigendecomposition of a kernel integral operator that optimize on a Dirichlet energey funciton are described in \cite{das2019delay, giannakis2019data}, and in eigendecompositions in reproducing kernel Hilbert spaces and as related to regularity are found in \cite{ klus2020eigendecompositions} also related to \cite{kevrekidis2016kernel}, amongst the growing literature in data-driven spectral analysis of transfer operators.  Our prior work in \cite{bollt2018matching} even developed a matching extended dynamic mode decomposition  (EDMD-M)  describing diffeomorphisms between dynamical systems directly in such a spectral framework.   Into this existing body of literature, we cast the results in this paper as oKEEDMD for an optimal Koopman eigenfunction interpretation of the extended DMD concept from evolution operators.  It describes that amongst the many eigenfunctions available as we note  may design an optimal representation functional so that both representation of evolution of observations is efficient, but it is done within a spectral framework, as KMD   is emphasized.


First, we review the underlying theory.  Consider,  a differential equation in $M\subset \mathbb{R}^d$,
\begin{equation}\label{ode}
\dot{x}=F(x),
\end{equation}
 with a vector field,
$F:M \rightarrow R.$  As usual, the nonautonomous scenario, $f(x,t)$ with $x\in {\mathbb R}^d$ can be written in $d+1$ dimensions as an autonomous problem, by augmenting with a time variable. 
The flow is stated  for each $t\in\mathbb{R}$ (or semi-flow for $t\geq 0$)  as a function, $x(t):=\varrho_t(x_0)$ for a trajectory starting at $x(0)=x_0\in M$. 

The associated Koopman operator,   often called a composition operator,  describes the evolution of ``observables", or ``measurements" along the flow, \cite{budivsic2012applied, mezic2013analysis}. Rather than analyzing individual trajectories in  phase space, {\color{black} we analyze  the dynamics of observations  functions.}
These ``observation functions",  
\begin{equation}
g:M\rightarrow {\mathbb C},
\end{equation}
 are elements of a space of observation functions ${\cal F}$.  Specifically, throughout this work, we will choose,
 \begin{equation}
 {\cal F}=L^2(M)=\{g:\int_M |g(s)|^2 ds<\infty \},
 \end{equation}
{\color{black} which} is commonly used since it is particularly convenient for numerical applications that utilize the inner product associated with the Hilbert space structure \cite{budivsic2012applied,kutz2016dynamic, mezic2005spectral, lan2013linearization, mezic2005spectral, mezic2019spectrum}.  There are many other useful choices for the function space, including other measures as a general $L^2(M,\mu)$ that allow notably for choices such as $\mu$ to be an invariant measure, or the SRB measure if one exists, and carefully associating the relevant domain.  The details of such choices effect the details of the spectral decomposition, but we will restrict this work to the Lebesgue measure, and also note that interpreting the composition operator in finite time makes for this simple choice to be especially useful when we turn to computational issues later in this paper.
 Also, for now we discuss scalar observation functions, but multiple scalar observation functions can be considered together, ``stacked" as a composite vector valued observation function.  
 {\color{black} Our eigenfunction discussion is stated in terms of smooth  (strong) solutions of a quasi-linear PDE, particularly as stated in Theorem \ref{thekoopmanpdetheorem}. {\color{black} Nonetheless, even if eigenfunctions so derived are  $C^1(U,[0,T])$, the discussion leading to infinitely many eigenfunctions for each eigenvalue holds in this subset and does not rely on the structure of $L^2$.}  Distinguishing these many eigenfunctions inspires our definition of primary eigenfunctions.  It is interesting to contrast this work to the discussion of principal eigenfunctions by Mohr-Mezic \cite{mohr2016koopman} considers the algebraic property of eigenfunctions to properly capture the so-called isostable coordinates, up to diffeomorphism; considering the Hartman-Grobman  linearization near an equilibrium,  eigenvalues and eigenfunctions are correspondingly matched to those of the full nonlinear Koopman operator, so accounting the  principal algebra.  Recently in Kvalheim-Revzen \cite{kvalheim2019existence}, showed a  $C^1$ existence and uniqueness theorem for principal eigenfunctions, and the resulting diffeomorphism to  isostable coordinates.  
 In the present work, we are able to describe cardinality of eigenfunctions beyond the principal algebraic structure, motivating our introduction of primary eigenfunctions, and we note that principal eigenfunctions can generate examples of our primary eigenfunctions, but generally there are more non-principle but primary eigenfunctions, as annunciated by Corollary \ref{uncount}.

 The dynamics of how observation functions change over time along orbits is what the Koopman operator defines.}
\begin{definition}\label{koopmandefnD}
{\bf Koopman Operator}. (Composition Operator), \cite{koopman1931hamiltonian, mezic2013analysis, cvitanovic2005chaos,gaspard2005chaos}. 
The operator, ${\mathscr K}_t$, associated with $\varrho_t$, is
 a (semi-) flow, stated as the following composition,
\begin{eqnarray}\label{koopmandefn}
{\mathscr K}_t:{\cal F}&\to& {\cal F}, \nonumber \\
g &\mapsto & {\mathscr K}_t[g](x)=g\circ \varrho_t,
\end{eqnarray}
 on the function space ${\cal F}$, for each $t\in\mathbb{R}$ (or as a semi-flow if the relation only holds for $t\geq 0$).
That is, for each $x$, we observe the value of an observable $g$ not at $x$, but ``downstream" by time $t$, at $\varrho_t(x)$.  See Fig.~\ref{Fig1}.  
\end{definition}


\begin{figure}[htbp]
\centering
\includegraphics[scale=1]{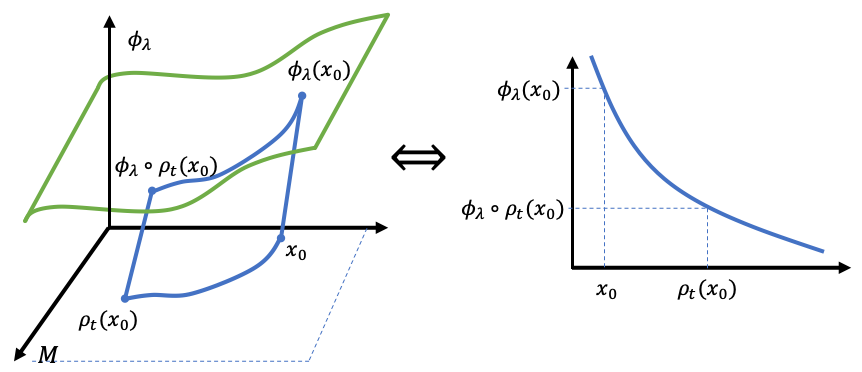}
 \caption{{\color{black}The action of a Koopman operator (composition operator) is to extract the value of a measurement function downstream, (Left) Eq.~(\ref{koopmandefn}), but (Right) for eigenvalue and eigenfunction pair (KEIGS) $(\lambda, \phi_\lambda(x))$ that function has the property Eq.~(\ref{eq:koopman definition}) effectively interpreting the change of $\phi_\lambda$ {\it along} an orbit, as if the dynamics is linear, even if the flow may be nonlinear in its phase space $M$. }}
\label{Fig1}
\end{figure}
An  interesting and important feature of the Koopman operator is that it is linear on $\mathcal{F}$, but at the cost of possibly being infinite dimensional, even though it may be associated with a flow $\varrho_t$ that evolves on a finite dimensional space, and indeed even due to a nonlinear vector field.
The spectral theory of Koopman operators \cite{gaspard2005chaos, budivsic2012applied, mezic2013analysis} concerns eigenfunctions and eigenvalues of the operator ${\mathscr K}_t$, which may be stated  in terms that the pair must satisfy the equation,
\begin{equation}\label{eq:koopman definition}
{\mathscr K}_t[\phi_\lambda](x)=b^t \phi_\lambda(x)=e^{\lambda t} \phi_\lambda(x).
\end{equation}
See Fig.~\ref{Fig1}.
We will write an eigenvalue---eigenfunction pair of the Koopman operator as $(\lambda,\phi_\lambda(x))$, and call this pairing a ``KEIG", and from here forward we will emphasize the association between such an eigenfunction and its eigenvalue, by the subscript, $\phi_\lambda(x)$. For convenience, we will say ``KEIGs" when referring to a Koopman eigenvalue and eigenfunction pair, and write, $(\lambda,\phi_\lambda(x))$. 
The multiplying factor follows the eigenvalue, $b=e^\lambda$.  {\color{black} A major point that we make herein is that for each $\lambda$, not only is the eigenfunction $\phi_\lambda$ not unique, there are  {\it at least uncountably infinitely many} eigenfunctions (but generally not linearly independent) associated with each $\lambda$, and we state this even while allowing only unit normalized eigenfunctions, to remove  the trivial idea that constant multiples of eigenfunctions are eigenfunctions.  }

A trending concept in the empirical study of dynamical systems has come to be the spectral decomposition of observables into eigenfunctions of the Koopman operator, 
Koopman mode decomposition (KMD, \cite{mezictraffic}).
  Let a vector valued set of observables ${\bf g}(x)=[g_1(x),...,g_D(x)]:M\rightarrow {\mathbb C}^D \in {\cal F}^D, $ be written as a linear combination of eigenfunctions,
\begin{equation}
    {\bf g}(x)=\sum_{j=1}^\infty \phi_{\lambda_i}(x) {\bf v}_j
\end{equation}
where the vectors ${\mathbf v_j}\in {\mathbb C}^D$ are called Koopman modes.  Further, the power of this concept lies in the following expression that describes the dynamics of observations in terms that remind us of the linear Fourier analysis, but now of the Koopman modes.  {\color{black}That is if, ${ y}(t):={ g}\circ { x}(t)\equiv { g}\circ \varrho_t({ x}_0):=K_t[{ g}](t) $, then, 
\begin{equation}
    { g}(t)=\sum_{j=1}^\infty  e^{\lambda_j t}\phi_{\lambda_j}({ x}_0){\mathbf v}_j.
\end{equation}
}

{\color{black} What is not generally discussed is the uniqueness of this decomposition, and furthermore the nature and even the cardinality of these eigenfunctions. Given nonuniqueness of the representation,  we are particularly interested here in {\it an efficient} representation.   Thus we see the nonuniqueness not as a problem, but rather as an opportunity to leverage toward better efficiency.    We will interpret this efficiency phrase  as is commonly used in the data driven concept leading to the POD-KL (Karhunen-Loeve) decomposition. A basis set should be designed for fast convergence, as compared against all other allowable basis sets, \cite{karhunen1947under, loeve1977elementary, webber1997karhunen, holmes2012turbulence, watanabe1969knowing}. } 
\begin{definition}\label{efficientr}
A finite  set of {\color{black} unit KEIGS $\{\varphi_{\lambda_i}(x)\}_{i=1}^k$, {\color{black} $\varphi_{\lambda_i}:U\rightarrow {\mathbb C}$,} $\|\varphi_{\lambda_i}\|_{L^2(U)}=1$ for an open domain $U\in M$ {\color{black} that is designed } empirically for a given target observation function, $q:U\rightarrow {\mathbb C}$ is {\bf a k-efficient }}  if it yields minimal error for a given finite partial sum:
\begin{equation}\label{effbasis}
    \min_{{\mathbf a}} \|\sum_{i=1}^k a_i \varphi_{\lambda_i}-q\|_{L^2(U)}\leq \min_{{\mathbf b}} \|\sum_{i=1}^k b_i \phi_{\tilde{\lambda}_i}-q\|_{L^2(U)}
\end{equation}
as contrasted against any other set of $k$-unit eigenfunctions, $\{\phi_{\tilde{\lambda}_i}(x)\}_{i=1}^k$.  This allows that even the eigenvalues may not necessarily match.
\end{definition}
We address these issues, in the setting of the underlying dynamics of a low-dimensional differential equation.   {\color{black} While definition Eq.~(\ref{effbasis}) could well be defined} without requiring the basis functions to be KEIGS, our point is by so restricting, we pursue an efficient spectral analysis within the KMD framework.  {\color{black} An eigenfunction representation offers an interpretation of a process in terms of spatial dynamically stationary modes.  Our exploitation of this efficiency concept  leads us to develop an optimal Koopman eigenfunction extended dynamic mode decomposition (oKEEDMD) in Sec.~\ref{numericalmethods}.}

The paper by Korda, Milan and Mezic, \cite{korda2019optimal} has a strong overlap with this work regarding Sec.~\ref{eigenfunctionPDEsec} in discussion of the PDE for Koopman eigenfunctions, and this part also overlaps our own previous work, \cite{bollt2018matching}, both of which  described eigenfunction solutions by the  method of characteristics.  {\color{black} We review} the form of the solution as a precursor to our interest here.  In  Sec.~\ref{Primarysec}, we define ``primary" eigenfunctions as a quotient identifying sets of eigenfunctions such that the level sets match, and this naturally relates to the algebraic property that generates eigenfunctions.  Also, this perspective leads to what we call a geometric observation that primary eigenfunctions compare in terms of transversality of how the level sets intersect, which we summarize in Theorem \ref{transversetheorem}.  In  \cite{korda2019optimal}, the authors pursue an existence of representation theorem for KMD expansions.  However, with our goal here toward efficient representation, and also understanding transversality between the levels sets of primary eigenfunctions, we proceed in Sec.~\ref{numericalmethods} to develop a numerical method to generate eigenfunctions that efficiently develop a KMD.  To do this, we bring together a numerical estimation associated with the eigenfunction PDE, {\color{black} exploiting the method of characteristics solution.}  This allows us to formulate a simple to solve optimization problem, posed initially on the data surface, and  then propagated simply along characteristics.  This leads to solutions that are eigenfunctions, which  successively lead to what we call a good dictionary of eigenfunctions. Thus we  define as our oKEEDMD algorithm.



 
\section{The Eigenfunction PDE, Spectrum, and Solutions}\label{eigenfunctionPDEsec}

{\color{black} The KEIGS follow the infinitesimal generator that gives a simple quasi-linear PDE that we discuss in this section.}
Corresponding to the statement of ${\mathscr K}_t$ as a semi-group of compositions, follows the action of the infinitesimal generator, \cite{klus2018data,mezic2019spectrum,lasota2013chaos,bollt2013applied},
\begin{equation}\label{ininitesimal}
     {\cal L}=F\cdot \nabla,
   \end{equation}
 and so we recall:
\begin{theorem}\label{thekoopmanpdetheorem}\cite{mezic2019spectrum,bollt2018matching}
  A smooth exact eigenfunction of the Koopman operator of a given flow, for a given eigenvalue $\lambda\in {\mathbb C}$ must satisfy the following PDE, 
 \begin{equation}  
\label{kooppde}
  F(x) \cdot \nabla \phi_\lambda(x) = \lambda \phi_\lambda(x),
 \end{equation}
 if {\color{black} $U\in M$ is compact, and $\phi_\lambda:U\rightarrow {\mathbb C}$, is in $C^1(U)$, or alternatively, if $\phi_\lambda\in C^2(U)$}.
\end{theorem}

The idea behind this theorem recalls the chain rule, $\dot{g}(x)$ for an orbit $x(t)=\varrho_t(x)$, so, (using the component notation, $x_i=[x]_i$ for the $i^{\mbox{th}}$ component of $x$),
\begin{equation}
\frac{d}{dt} {\phi(x(t) )}=\sum_i \frac{\partial \phi}{\partial x_i} \frac{d x_i}{d t}=\sum_i \frac{\partial \phi}{\partial x_i} F_i(x)= \nabla \phi \cdot F(x).
\end{equation}
However, the proof refers to the infinitesimal generator.
{\color{black} This PDE  is classified as a quasi-linear,} and the solution thereof is well handled by the method of characteristics as discussed in many classical textbooks on PDEs, \cite{john1975partial}, {\color{black}which we review in this context in Appendix \ref{characteristics}.  }  {\color{black} That the initial data and the domain $U$ play an important role when stating eigenfunctions derived by this PDE is emphasized with a simple example in Appendix \ref{fixapp}.  These eigenfunctions of the infinitesimal generator relate to the term {\it open eigenfunction} (introduced in \cite{mezic2019spectrum} and we used in \cite{bollt2018matching}), for a finite nonzero time step, $t>0$ of the flow $\varrho_t$.   This demands of the composition operator statement of an eigenfuncton, Eq.~(\ref{eq:koopman definition}), a subdomain in which both $x$ and $\varrho_t(x)\in U$.}

In characterizing the set of eigenfunctions, it is useful to recall some of the standard spectral theory of linear operators.

\begin{remark} From spectral theory we review, \cite{yosida1994functional, boccara1990functional,kadison1983fundamentals}, the resolvent set of a linear operator ${\cal L}:{\cal F} \rightarrow {\cal F}$ consists of those functions in a Hilbert space ${\cal F}$ for which the resolvent exists.  That is, when the  operator,
\begin{equation}
    L_\lambda({\cal L})=({\cal L}-\lambda I), \mbox{ as a function of }\lambda\in {\mathbb C},
\end{equation}
has a bounded inverse, and when that operator exists call it, ${\cal R}_\lambda({\cal L})$.  The Koopman PDE, Eq.~(\ref{kooppde}), may be stated in terms of the resolvent operator, as 
\begin{equation}
L_\lambda({\cal L})\phi_\lambda(x)=(F\cdot \nabla -\lambda I)\phi_\lambda(x)=0.
\end{equation}

\begin{definition}\label{defn2}
The set of all complex numbers $\lambda$ such that the resolvent operator ${\cal R}_\lambda({\cal L})$ 1) exists, 2) is bounded, and 3) is densely defined is called the resolvent set, $\rho({\cal L})\subset {\mathbb C}$.  The complement is called the spectrum of ${\cal L}$, denoted $\sigma({\cal L}).$   
By complimentary definition, the type of failure mode according to each of each of these properties classifies the type of spectrum: 
\begin{enumerate}
    \item If
${\cal R}_\lambda({\cal L})$ fails to exist, as there is no inverse for those $\lambda$, then declare the {\it point spectrum}, $\lambda\in P_\sigma({\cal L})\subset {\mathbb C}$. 
\item If ${\cal R}_\lambda({\cal L})$ exists and it is densely defined, but is not bounded, then declare the corresponding $\lambda\in C_\sigma({\cal L})$, the {\it continuous spectrum}.  
\item If ${\cal R}_\lambda({\cal L})$ fails to be densely defined  for a given $\lambda$, then declare the {\it residual spectrum}, $\lambda\in R_\sigma({\cal L})$.  
\end{enumerate}
\end{definition}
Clearly, $\sigma({\cal L})=P_\sigma({\cal L}) \cup C_\sigma({\cal L})\cup R_\sigma({\cal L})$. 
{\color{black} Now  to the point of this remark: } the names of these spectral sets are not to be misinterpreted as indicating cardinality of the sets,  \cite{boccara1990functional}.   Appendix \ref{ctsspect} describes the role of the domain of the dynamical system and the consequences to the spectral decomposition.

\end{remark}

{\color{black} Solutions of the Koopman PDE follow as corollary to Theorem \ref{thekoopmanpdetheorem}, the details of which are proved in the Appendix \ref{characteristics}, by application of the method of characteristics, \cite{john1975partial} as in context in \cite{bollt2018matching}.}  Extending to larger domains  demands nonrecurrence was noted in \cite{korda2019optimal} to prevent the complication of defining multi-valuedness to eliminate consideration of loops such as periodic orbits as well as orbit segments that exit and re-enter the domain of interest. 

\begin{theorem}\label{gensolkoop}
Let the Koopman eigenfunction PDE, Eq.~(\ref{kooppde}) be defined for an ODE, $\dot{x}=F(x)$, with a flow $x(r)=\varrho_r(x_0):M\times {\mathbb R}\rightarrow M$.  Assume a closed {\color{black} co-dimension-one} initial (data) manifold $\Lambda \subset M$ that is nonrecurrent for some time epoch, $r\in [t_1,t_2]$, that contains $0$, {\color{black} and transverse to the flow,} and let $U=\cup_{t\in [t_1,t_2]} \varrho_t(\Lambda)$ be the resulting nonrecurrent closed domain.  Furthermore, let a $C^1(\Lambda)$ initial data function be defined, $h:\Lambda\rightarrow {\mathbb C}$.  (See Fig.~\ref{pullback}). Then an open-KEIGS pair, $(\lambda,\phi_\lambda(x)), \phi_\lambda:U\rightarrow {\mathbb C}$ has the form,
\begin{equation}\label{gensol}
    \phi_\lambda(x)=h\circ s^*(x)e^{\lambda r^*(x)},  
\end{equation}
where $r^*(x)$ is the ``time"-of-flight such that for a point $x\in U$,  
there is an intersection in $U$ by pull back to the data surface $\Lambda$, 
\begin{equation}\label{rstar}
r^*(x)= \{r: \varrho_{-r}(x)\cap \Lambda\neq \emptyset\}.
\end{equation}  
For  each  $x\in U$, 
\begin{equation}\label{sstar}
    s^*(x)=s\circ \varrho_{-r^*(x)}(x),
\end{equation} {\color{black} where $s$ is the parameterization on $\Lambda$. Overloading  notation in naming a function, $s(y)$ states the parameterization for a point $y\in U\in \Lambda$, and  finally $s^*:U \rightarrow \Lambda$ is the position on $\Lambda$ of the pull back from $x$ to  that first intersection point on $\Lambda$.}
\end{theorem}

See Fig.~\ref{pullback} as a description of an ``open"-eigenfunction \cite{bollt2018matching,mezic2017koopman} which is easily understood as resulting  from propagating the data by the method of characteristics.  {\color{black} For clarity, our examples and pictures illustrate  $M \subset {\mathbb R}^2$ and the co-dimension-one data manifold $\Lambda$ is a curve.} Since  finite domains are sufficient,  where the flow may leave the domain in finite time, an open-eigenfunction is the composition operator analogue of an eigenfunction of the Koopman operator. 
{\color{black} Understanding that,  $\Lambda \subset int(U)$, then with $t_1<0<t_2$ defines an eigenfunction that extends forward in time (upstream) and backward in time from $\Lambda$.}
  Notice that by asserting nonrecurrence of $\Lambda$, that $U$ is constructed to be nonrecurrent.  

\begin{proof}
{\color{black}
In Appendix Sec.~\ref{characteristics}, we review the standard method of characteristics \cite{john1975partial} for quasi-linear PDE's including the Koopman PDE Eq.~(\ref{kooppde}), which is the central to the derivation of Eq.~(\ref{gensol}).    In brief, a problem of the form, $a(x_1,x_2,u)u_{x_1} + b(x_2,x_2,u) u_{x_2}=c(x_1,x_2,u),\mbox{ } (x_1,x_2)\in \Omega$ can be interpreted as $\nabla U \cdot <a,b,-c>=<u_{x_1},u_{x_2},-1>\cdot<a,b,-c>=0$.  So transversally to the surface $\Lambda$, $\nabla U$  is  also perpendicular to the vector field $<a,b,c>$ that defines characteristic curves which are therefore solutions of an ODE, $\frac{dx_1}{dr}=a(x_1,x_2,z), \frac{dx_2}{dr}=b(x_1,x_2,z), \frac{dz}{dr}=c(x_1,x_2,z),$  whose initial conditions we will choose on the data surface $\Lambda$ that is parameterized by $\Lambda$.  The many variable version of this statement is similar.  
The rest of the proof is in Appendix \ref{characteristics}, Eqs.~(\ref{pf1})-(\ref{pf2}).    In brief, the solution of the Koopman PDE implies that initial data emanates along characteristics that are solutions of the flow.  }
See Fig.~\ref{pullback}, where we illustrate that a general solution of the eigenfunction PDE is simply a pull back along the flow through $x$, to read the data on $\Lambda$, and then scale it according to the linear action of $(e^\lambda)^r$, for the ``time" it takes to pull back the point $r=r^*(x)$.   
From $x\in U$, ``pull"-back along the flow from $x$ to $\Lambda$ is to where initial data $h$ is defined.  Therefore, the value of $\phi_\lambda$ at $x$ is $h$ but scaled  as a function of the  time of the pullback: $e^{\lambda r}$.  
\end{proof}

\begin{figure}[htbp]
\centering
\includegraphics[scale=1]{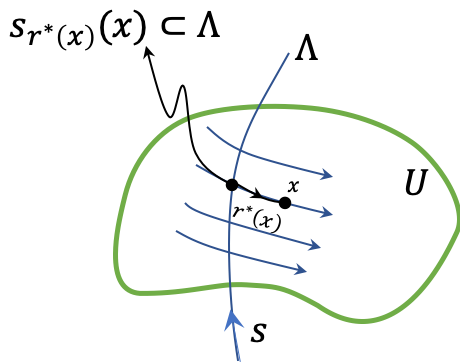}
\caption{\label{pullback} A general {\color{black} open eigenfunction,} Eq.~(\ref{gensol}) solution of Eq.~(\ref{kooppde}) is defined in terms of measuring the initial data $\protect{h(s)}$ on the data surface $\protect{\Lambda}$, and for each $\protect{x}$ there is a unique point where that measurement is taken.  Then the solution $\protect{\phi_\lambda}$ is a linear scaling of that measurement, by the time of the pull back.}
\end{figure}

\begin{remark}
Given a point $x\in U$, finding $r^*(x)$ is a task that is computationally comparable to a shooting problem, pulling $x$ back to the data curve $\Lambda$ along the flow, and recording the ``time"-of-flight.
For some problems where the flow is known in closed form, this may be solvable analytically.  On the other hand, a ``Poincare'" section type solver can be developed to numerically continue from initial condition $x$ by the backward flow,
$    \dot{x}=-F(x).$ From this follows $s^*(x)$ by Eq.~(\ref{sstar}).   {\color{black} So the eigenfunction can be stated in closed form when the solution of the flow is known in closed form, or otherwise, but numerical shooting and event stopping using a numerical integrator, it is defined pointwise and numerically.}
\end{remark}

The simple form of the eigenfunction solutions, $\phi_\lambda(x)$, Eq.~(\ref{gensol}), {\color{black} also has a simple geometric interpretation.}  Notice that since $s^*(x)$ is constant along the flow, then $h\circ s^*(x)$ is also constant along the flow.  Therefore,
\begin{equation}
    \phi_\lambda(x) e^{-\lambda r^*(x)}=h\circ s^*(x).
\end{equation}
These functions eigenfunctions, $\phi_\lambda$ have level sets which when compared to an orbit, evolve as simply  summarized:
\begin{corollary}
If $x$ and $\tilde{x}$ are on the same orbit in $U$, so that $x=\varrho_r(\tilde{x})$,
then,
\begin{equation}\label{res}
    \phi_\lambda(x)=\phi_\lambda (\tilde{x})e^{\lambda r},
\end{equation}
where $r$ is the ``time" to flow from $\tilde{x}$ to $x$.   
\end{corollary}
This statement reduces to the general form Eq.~(\ref{gensol}), on the data set $\tilde{x}\in \Lambda$.  
The dynamics evolves the eigenfunction by linear scaling of initial data, allowing when real valued, monotone increasing or decreasing, or when complex valued eigenvalues, this linear scaling includes oscillations.  In this sense, the observable dynamics of an eigenfunction, are especially simple as defined by the level sets of $\phi_\lambda(x) e^{-\lambda r^*(x)}$.
This statement requires that we recall a definition of level sets for a general function, $f:U\rightarrow V$. We use the notation for a level set, of level $c$ in the range $V$, from the domain $U$ to be the set,
\begin{equation}\label{levelset}
    L_c(f)=\{x:f(x)=c, x\in U\}.
\end{equation}
If $c\notin V$, then $L_c(f)=\emptyset$.

The following corollary describes  a general principle that initial data curves, and therefore the data defined on them, {\color{black} can otherwise be stated more conveniently elsewhere in the domain.}
A seemingly complicated initial data functions, $h:\Lambda \rightarrow {\mathbb C}$ may correspond to another data curve $\tilde{\Lambda}$, its image under the flow and corresponding data function on it, $\tilde{h}(\tilde{s})$, {\color{black} an allusion to Remark \ref{r3}.}  $s$ describes a parameterization on $\Lambda$ and $\tilde{s}$ describes a parameterization on $\tilde{\Lambda}$, so that corresponding points $\Lambda_s$ flows to $\tilde{\Lambda}_{\tilde{s}}$ under the flow, $\varrho$.  Specifically,

\begin{corollary}\label{principle} If $h(s):\Lambda \rightarrow {\mathbb C}$ data is defined on $\Lambda\in U$, then the corresponding equivalent data on $\tilde{\Lambda}\in U$ is a function $\tilde{h}(\tilde{s}): \tilde{\Lambda} \rightarrow {\mathbb C}$ if corresponding points $\tilde{x} \in \tilde{\Lambda}$ flow to $x \in \Lambda$, $x=\varrho_{r^*(x)}(\tilde{x})$.  With the same smoothness as the flow, there exists a transformation, $\alpha:  \tilde{\Lambda}\rightarrow \Lambda$
such that $\tilde{h}(\tilde{s})=h\circ \alpha (s)$.
\end{corollary}
This is just a reinterpretation of Eq.~(\ref{res}).  The idea is, if $\tilde{\Lambda}$ is an image of $\Lambda$ under the flow, then the result follows if $\tilde{h}$ on $\tilde{\Lambda}$ corresponds accordingly to $h$
on $\Lambda$.
Consider that for each $x\in U$, then $(r^*(x),s^*(x))$ has already been defined as the pull back time and parameterized position on $\Lambda$, and $h$ can be evaluated there.  Or likewise, $(\tilde{r}(x),\tilde{s}(x))$ are the pull back onto $\tilde{\Lambda}$.    

\begin{remark} \label{r3} {\color{black} A simple interpretation following Corollary 1 and 2 is that there is a ``dual" characterization of an eigenfunction.  Either we can focus on an arbitrary transverse manifold $\Lambda$, and define a data function $h$ on $\Lambda$, with the eigenfunction $\phi_\lambda$ that follows,  Eq.~(\ref{gensol}).  Alternatively, for that same eigenfunction $\phi_\lambda$, we can focus on  a given $h$ is a constant $h=c$, and then $\Lambda$ is the corresponding level set of $\phi_\lambda|_\Lambda=c$.}
\end{remark}

To illustrate that data curves $\Lambda$ and how the data functions defined on them, $h(s)$ evolves under the flow, consider the linear example worked out extensively in the Appendix, Eqs.~(\ref{ceigs}), (\ref{ceigs2}) from a  linear differential equation system. Eigenfunctions are explicitly derived in the Appendix from two different transverse surfaces. One data surface $\Lambda $ is chosen to be a horizontal line, and then the other $\tilde{\Lambda}$ is chosen to be a vertical line.  In the first case, we found the observer function $(a_2,x_2)$ by special choice $\lambda_2=a_2$ and a constant data function, but  in the second case, $(a_1,x_1)$ with $\lambda_1=a_1$.  See Fig.~\ref{figcb}.
\begin{equation}\label{pull1}
    \tilde{s}=s^{-\frac{a_2}{a_1}},
\end{equation}
and so,
\begin{equation}\label{pullbackcurve}
    \tilde{h}(\tilde{s})=h(s^{-\frac{a_2}{a_1}})s^{\frac{\lambda}{a_1}}.
\end{equation}
defines the data on a vertical data curve $\Lambda$ from data originally stated on a horizontal curve $\Lambda$.  
\begin{figure}[htbp]
\centering
\includegraphics[scale=1]{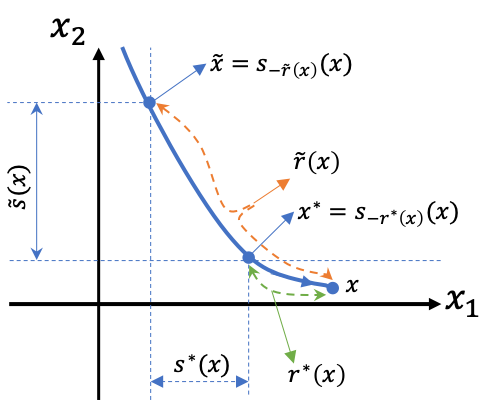}
\caption{   Data $h$ stated on one data curve $\Lambda$ can be pulled back to restated as $\tilde{h}$ on some other data curve $\tilde{\Lambda}$ here shown as the Example Eqs.~(\ref{pull1})-(\ref{pullbackcurve}) demonstrating the principle of Corollary \ref{principle}.
\label{figcb}
}
\end{figure}


\section{Algebraic Structure and Cardinality of Eigenfunctions}

There are at least two major reasons for the fact of uncountably many eigenfunctions.
First there is a simple and well known,  algebraic property of eigenvalues and eigenfunctions as follows.  But this is not the exclusive source of these many eigenfunctions.
\begin{proposition}  Algebraic Property.  The  KEIGS of a Koopman operator of a flow form a semigroup as follows.  For any $0<\alpha_1,\alpha_2<\infty$, and pair of KEIGS $(\lambda_1,\phi_{\lambda_1}(x))$, $(\lambda_2,\phi_{\lambda_2}(x))$  in an open domain $U$, then,
\begin{equation}\label{algprop}
(\lambda,\phi_\lambda(x))=((\alpha_1 \lambda_1+\alpha_2 \lambda_2),(\phi_{\lambda_1}(x))^{\alpha_1}(\phi_{\lambda_2}(x))^{\alpha_2}),
\end{equation}
 is also a KEIGS if $\phi_\lambda(x)$ exists in $U$. 
\label{algebraicprop}
\end{proposition}
A proof is found in Appendix \ref{pflin}, that appeals directly to the PDE, but other approaches are found elsewhere, such as \cite{budivsic2012applied}.

As a direct consequence,  powers of a KEIGS are also KEIGS as we see by choosing, $\phi_1=\phi_2$  in  Proposition \ref{algebraicprop}.
If $(\lambda,\phi_\lambda(x))$ is a KEIGS, then for any $s>0$, then $(\lambda^s,(\phi_\lambda(x))^s)=(\lambda^s,\phi_{\lambda^s}(x))$ is also a KEIGS.
Consider a special case, where $\phi_1=\phi_2=1$ and we see that $\lambda_1\lambda_2$ is the eigenvalue of the product of eigenfunctions, $\phi_{\lambda_1}(x)\phi_{\lambda_1}(x)$.   Even more specialized, with the same eigenfunction, we see that $\lambda_1^2$ is an eigenvalue of $\phi_{\lambda_1}(x)^2$. Squares of eigenfunctions, and likewise  any other power for that matter, are eigenfunctions.

\begin{remark} {\bf Example.}
In
Appendix \ref{characteristics} example \ref{linearexample}, follows that   $\dot{x}=a x$, has the state observer as KEIGS:  
$(\lambda,\phi_\lambda(x))=(a,x).$
Therefore, by Proposition \ref{algebraicprop}, 
the following are also examples of KEIGS, $(\lambda,\phi_\lambda(x))=(2a,x^2$), $(\lambda,\phi_\lambda(x))=(3a,x^3$), etc., as well as $(\lambda,\phi_\lambda(x))=(\frac{1}{2} a,\sqrt{x})$, etc, in appropriate domain for $x$.  All of these follow by exponentiation directly from Proposition \ref{algebraicprop}, but also this is in agreement with the direct derivation of the KEIGS in Eq.~(\ref{gensol}) that yields, for $h=1$, $\phi_\lambda(x)=c x^{\frac{\lambda}{a}}$ for any $\lambda\in {\mathbb C}$.


\end{remark}

The following is immediate due to the algebraic property associating KEIGS with exponentiation.
\begin{corollary}\label{uncount}
The cardinality of KEIGS is at least uncountable.  
\end{corollary}

However, this algebraic property generating eigenfunctions is not the only source counting ``new" KEIGS, and furthermore, in terms of a sensible dimensional definition of new, that is ``primary", these are {not really new}.  {\color{black} In other words, principle eigenfunctions do not account for all of the primary eigenfunctions of KEIGS}.  We will define in the next section what we mean by {\it not really new}.  Another source of KEIGS due to freedom in choosing data functions on transverse curves generates  a cardinality  greater than the uncountability generated by the algebraic property just reviewed. We discuss these in the next section.


\section{Definition of Primary Eigenfunctions to Complete the Geometry }\label{Primarysec}

While the algebraic property generates at least uncountably many KEIGS, we have argued recently \cite{bollt2018matching} and likewise others have pointed out \cite{korda2019optimal}, there is geometry encoded in the eigenfunctions.  
{\color{black} To this end we  define  {\it primary eigenfunctions,} which is related  to ``principal" eigenfunctions as described in Mohr-Mezic, \cite{mohr2016koopman}, in that principal eigenfunctions account for the algebraic structure, and principal eigenfunctions are examples that generate  primary eigenfunctions, but there are more primary eigenfunctions that the discussion of principal does not emphasize.}


Our definition of primary KEIGS is in terms of the set of level sets, Eq.~(\ref{levelset}), of eigenfunctions:
\begin{definition}\label{defnequiv}
(Primary KEIGS). For any one KEIGS pair, 
$ ( \lambda, \phi_{\lambda}(x)) $, 
$ \lambda\in \mathbb{C}, x\in U\subset  M$, we may form an equivalence relationship to any other KEIGS pair $(\tilde{\lambda},\tilde{\phi}_{\tilde{\lambda}}(x))$, $\tilde{\lambda}\in \mathbb{C}$, defined in the same domain, $x\in U$, that share the same set of level sets in $U$.  That is, write 
$ (\lambda,\phi_{\lambda}(x)) \doteq (\tilde{\lambda},\tilde{\phi}_{\tilde{\lambda}}(x))$ iff in terms of the set of level sets, for each $c\in V$, there is a $\tilde{c} \in V$ such that $L_c(\phi_\lambda)=L_{\tilde{c}}(\tilde{\phi}_{\tilde{\lambda}})$.  That is the level sets match between the functions, although not necessarily corresponding to the same levels.  Then, continuing to refer in terms of any one specific KEIGS, $(\lambda,\phi_\lambda(x))$, we define an equivalence class of  KEIGS,
\begin{equation}
    \overline{ (\lambda,\phi_\lambda(x)) } :=\{ (\tilde{\lambda},\tilde{\phi}_{\tilde{\lambda}}(x)):  (\tilde{\lambda},\tilde{\phi}_{\tilde{\lambda}}(x))\doteq (\lambda,\phi_\lambda (x)) \}
\end{equation}
That is, the set of KEIGS that are ``$\doteq$" equivalent to a specific chosen (convenient) member $(\lambda,\phi_\lambda(x))$ of the equivalence class is a quotient set of functions.
\end{definition}



There is a geometric statement that justifies this equivalence class, as follows.
\begin{theorem}\label{alggeothm}
Given a KEIGS, $(\lambda, \phi_\lambda(x))$, then  other KEIGS shares the same set of level sets if it is derived  by exponentiation following the algebraic property Eq.~(\ref{algprop}). Therefore,  the algebraic property of eigenfunctions {\color{black} yields a} primary KEIGS, at least in terms of inclusion allowing that general exponentiation may not exist.
\begin{equation}\label{inclusion}
    \{(\lambda^p,({\phi_\lambda}(x))^p,p\in {\mathbb R}\}\subset 
\overline{(\lambda, \phi_\lambda(x))}.
\end{equation}
\end{theorem}
\begin{proof}
Let,
\begin{equation}
    L_c(\phi_\lambda)=\{x: \phi_\lambda(x)=c, x\in M\},
\end{equation}
be the $c$ is a constant level set for a fixed $c\in {\mathbb C}$ in the range of $\phi_\lambda(x)$.  Then, 
\begin{equation}
    L_{c^s}(\phi_\lambda^s)=\{x: (\phi_\lambda(x))^s=c^s, x\in M\},
\end{equation}
is the $c^s$ constant level set of the eigenfunction $(\phi_\lambda(x))^s$ which is also an eigenfunction by the algebraic property. Algebraic manipulation clarifies that $L_c\subseteq L_{c^s}$, with equality at least when $c$ is positive, or $s$ is an even integer. 
\end{proof}

\begin{remark}
For convenience, we use the phrase ``primary KEIGS" to refer to the equivalence class of functions (that is the set of functions), but also for simplicity, sometimes we use the same phrase ``primary KEIGS" when referring to  a single eigenvalue, eigenfunction (KEIGS) pair that can be used to generate the entire set by exponentiation.  We do this especially when describing a favorite (simple) such pair.  

For example, consider the linear example, $\dot{x}=a x$, Appendix \ref{characteristics}, has in particular the primary KEIGS,  \begin{equation}\label{inclusion2}
\overline{(a,x)}=\{(a^p,x^p):p\in {\mathbb R}\}.
    \end{equation}
A particularly convenient member of this set follows the choice, $\lambda=a$.  In this case, we call the corresponding eigenfunction $\phi_a(x)=x$ as ``the state observer" since it is the identity function.  For convenience we ``overload the phrase," and drop the overline if stating that $(a,x)$ is a primary KEIGS, also referring to the set of equivalent  KEIGS.  See Fig.~\ref{Fig5}, where we show several equivalent KEIGS in this primary KEIGS.
\begin{figure}[hbtp]
\centering
\includegraphics[scale=0.17]{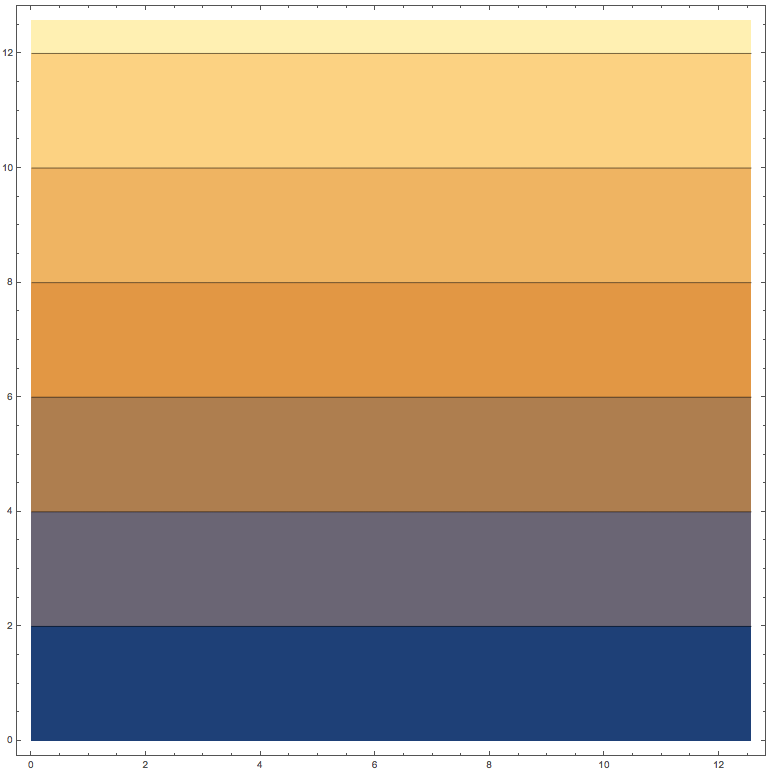}
\includegraphics[scale=0.17]{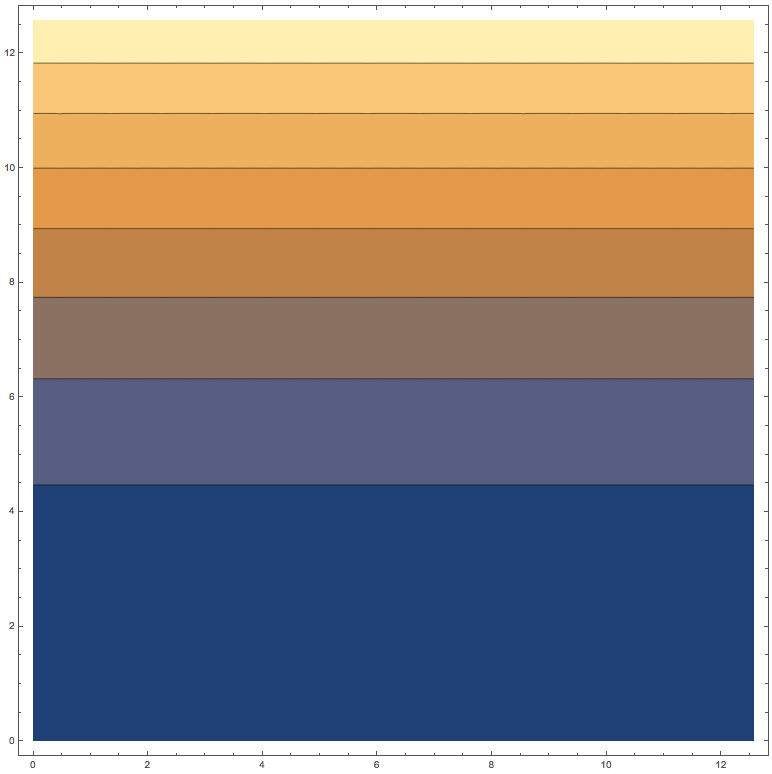}
\includegraphics[scale=0.17]{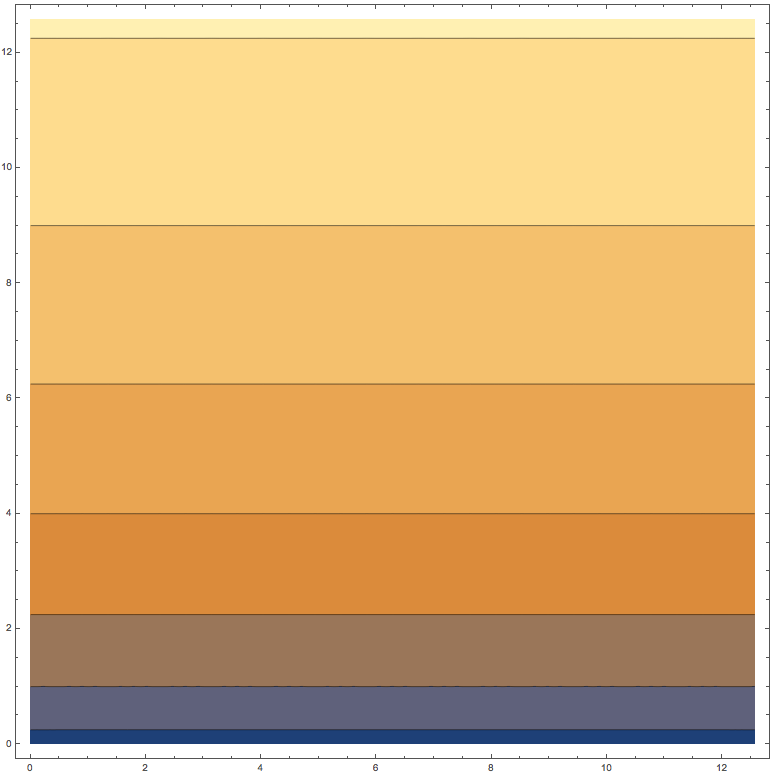}
\caption{ The algebraic property yields a family of KEIGS according to Theorem \ref{alggeothm}.  Here, from the linear system, Eq.~(\ref{slin}), results equivalent KEIGS, $(a_2,x_2), (a_2^2 x_2^2), (\sqrt{a_2},\sqrt{x_2})$, shown left to right, as corresponding contour plots.  Notice that the set of level sets are the same and in this sense we define the eigenfunctions  equivalence, Eqs.~(\ref{inclusion})-(\ref{inclusion2}), as defined according Definition \ref{defnequiv} of primary KEIGS.  Contrast this figure to Fig.~\ref{Fig4}.}
\label{Fig5}
\end{figure}
\end{remark}

Now to the title of this section, that Primary KEIGS complete the geometry, {\color{black} in the sense that the algebraic structure alone does not account for all of the set of level sets, motivates the following example.}

\begin{remark}{\bf Example 1: Non-Primary KEIGS of a Linear System, Despite Same Eigenvalues.}
In the Appendix, Eqs.~(\ref{ceigs}), (\ref{ceigs2}), we showed that the linear system,
\begin{equation}\label{slin}
\dot{x_1}=a_1 x_1, \dot{x_2}=a_2 x_2,
\end{equation}
has a KEIGS in general form $\phi_\lambda (x_1,x_2)=x_2^{\frac{\lambda}{a_2}}h\left(\frac{x_1}{x_2^{(\frac{a_1}{a_2})}}\right),$ from   an initial data curve $\Lambda$ as a horizontal line, and this class includes the observer function, KEIGS, $(a_2, x_2)$ when the initial data is chosen to be a constant, $h(s)=1$.  However, if the initial data on $\Lambda$ is not chosen as a constant, say for example,  as $h(s)=s$, then 
\begin{equation}
    \phi_\lambda (x_1,x_2)=x_1x_2^{\frac{\lambda-a_1}{a_2}},
\end{equation}
results.  Or using the same eigenvalue, $\lambda_2=a_2,$
\begin{equation}
    \phi_{a_2} (x_1,x_2)=x_1x_2^{\frac{a_2-a_1}{a_2}}.
\end{equation}
To see that these are {\bf not} part of the same primary KEIGS, despite the {\bf same} eigenvalues and same initial data curve, but  different data functions are used, we must check that these result in different sets of level sets.  Since the gradient is perpendicular to the level sets, and the perp-gradient points along level sets, we can check that,
\begin{equation}\label{orth}
    \nabla  ( x_1x_2^{\frac{a_2-a_1}{a_2}}) \cdot \nabla_\perp(x_2) =<x_2^{\frac{a_2-a_2}{a_2}},\frac{(a_2-a_1)}{a_2x_1x_2^{\frac{a_1}{a_2}}}>\cdot <1,-0>=x_2^{\frac{a_2-a_1}{a_2}}:\neq 0.
\end{equation}
For example, if the system were such that $a_1=1, a_2=2$, then this dot product is clearly spacially varying, as $\sqrt{x_2}$.  So transversality is generally easy to construct even for simple examples.  See Fig.~\ref{Fig4}.
\begin{figure}[hbtp]
\centering
\includegraphics[scale=0.3]{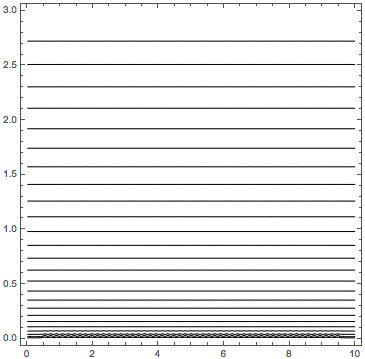}
\includegraphics[scale=0.3]{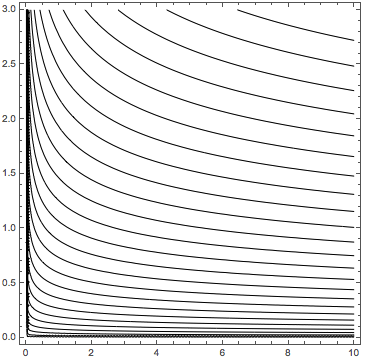}
\includegraphics[scale=0.3]{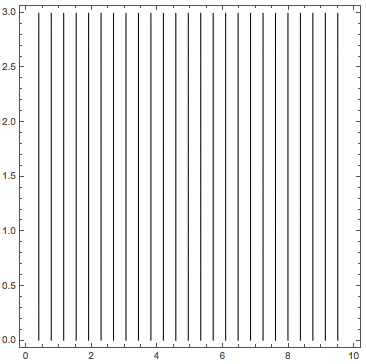}
\includegraphics[scale=0.3]{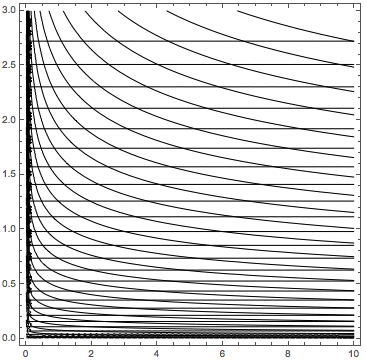}
\caption{ Non-equivalent KEIGS.  Following Example 1, KEIGS due to same eigenvalue and same initial data set $\Lambda$ but different initial data functions, $h, \tilde{h}: \Lambda \rightarrow {\mathbb C}$ result in KEIGS whose level sets are not equivalent, and therefore they generate different primary KEIGS.  Here clearly non-equivalent KEIGS result from using the simple linear system, Eq.~(\ref{slin}), and the same $\Lambda$ is the horizontal line at $x_2$=1, and the same eigenvalue, $\lambda=a_2$, {\it but differing data functions.}  Left uses $h(s)=1$ yielding $\phi_{a_2}(x_1,x_2)=x_2$ (that is the eigenfunction usually discussed, and it is usually called the ``observer" function), but next left to right result from data $h(s)=s^2$, $h(s)=s$ yielding $\phi_{a_2}(x_1,x_2)=x_1x_2^{\frac{a_2-a_1}{a_2}}$, and $\phi_{a_2}(x_1,x_2)=x_1^2x_2^{\frac{a_2-2a_1}{a_2}}$, (and if for example $a_2=2, a_1=1$, then specifically $\phi_{a_2}(x_1,x_2)=x_1\sqrt{x_2}$, and $\phi_{a_2}(x_1,x_2)=x_1^2$).  Observe the differing level sets, as visually we see clearly by this geometry that $\nabla  ( x_1x_2^{\frac{1}{2}}) \cdot \nabla_\perp(x_2) \neq 0$ in Eq.~(\ref{orth}).  The right figure shows the transverse intersections of the first two eigenfunctions on the left that were due to $h(s)=1$ and $h(s)=s^2$.}
\label{Fig4}
\end{figure}

\end{remark}

\begin{remark}{\bf Example 2: Non-Primary KEIGS of a Nonlinear System}

Consider the following ODE in the plane, 
\begin{eqnarray}\label{hopfn}
\dot{x_1}&=&-x_2 + x_1(\mu - x_1^2 - x_2^2) \nonumber \\
\dot{x_2}&=&x_1 + x_2(\mu - x_1^2 - x_2^2),
\end{eqnarray}
which is a well known ODE as it is often used as a normal form to present the unfolding of a Hopf bifurcation, \cite{perko2013differential}. Polar coordinates conveniently follows the change of variables, $x_1=r cos(\theta), x_2=r sin (\theta)$, 
\begin{eqnarray}
\dot{r}&=&r(\mu-r^2), \nonumber \\
\dot{\theta}&=&1.
\end{eqnarray}
The supercritical Hopf bifurcation occurs at $\mu=\mu_c=0$, whereafter, $\mu>\mu_c$ there is a stable limit cycle.  In Fig.~\ref{shoota}, we show a stream plot for many initial conditions in and around the limit cycle, choosing $\mu=1$.  This makes a good case to study the numerics of using the Koopman PDE to find eigenfunctions by a numerical integration method.  We can contrast the numerical solution to closed form computation as this nonlinear problem is still analytically solvable.  By separation of variables, using initial condition 
$r(t_0)=r_0, \theta(t_0)=\theta_0$,
\begin{equation}
t-t_0=\frac{2 \ln r (\ln r_0^2-\mu) }{ \ln r_0(r^2-\mu)},
\end{equation}
which if choosing, $t_0=0$,  gives, 
\begin{equation}
    r(t)=\frac{e^{\mu t} \sqrt{\mu} r_0}{\sqrt{\mu - r_0^2+e^{2\mu t}r_0^2}}.
\end{equation}

To state an eigenfunction by Theorem \ref{gensolkoop}, we must choose an initial data curve $\Lambda$ and there are many possible choices leading to different domains with of course specific KEIGS depend on the domain. Choosing for example, $\Lambda$ to be a circle of fixed radius centered on the origin, but $R>\sqrt{\mu}$, so larger than the radius of the limit cycle, then the eigenfunction is,
\begin{equation}
\phi_\lambda(r_0,\theta_0)= h(\theta_0+t) e^{\frac{2\lambda \ln R (\ln r_0^2-\mu) }{ \ln r_0(r^2-\mu)}}.  
\end{equation}
Given a cartesian point initial point, $(x,y)$, we compute the corresponding initial $(r_0,\theta_0)$ as usual, $r_0=\sqrt{x^2+y^2}$, $\theta_0=\tilde{tan^{-1}}(x,y)$ (using the often used 4-quadrant aware arctan2  function which we write as $\tilde{tan^{-1}}$ rather than the ``common" $tan^{-1}$ function), and so,
\begin{equation}
\phi_\lambda(x_1,x_2)= h(\tilde{tan^{-1}}(x_1,x_2)+t) e^{\frac{2\lambda \ln R (\ln (x_1^2+x_2^2)-\mu) }{ \ln \sqrt{x_1^2+x_2^2} (x_1^2+x_2^2-\mu)}}.  
\end{equation}

\begin{figure}[htbp]
\centering
\includegraphics[scale=0.2]{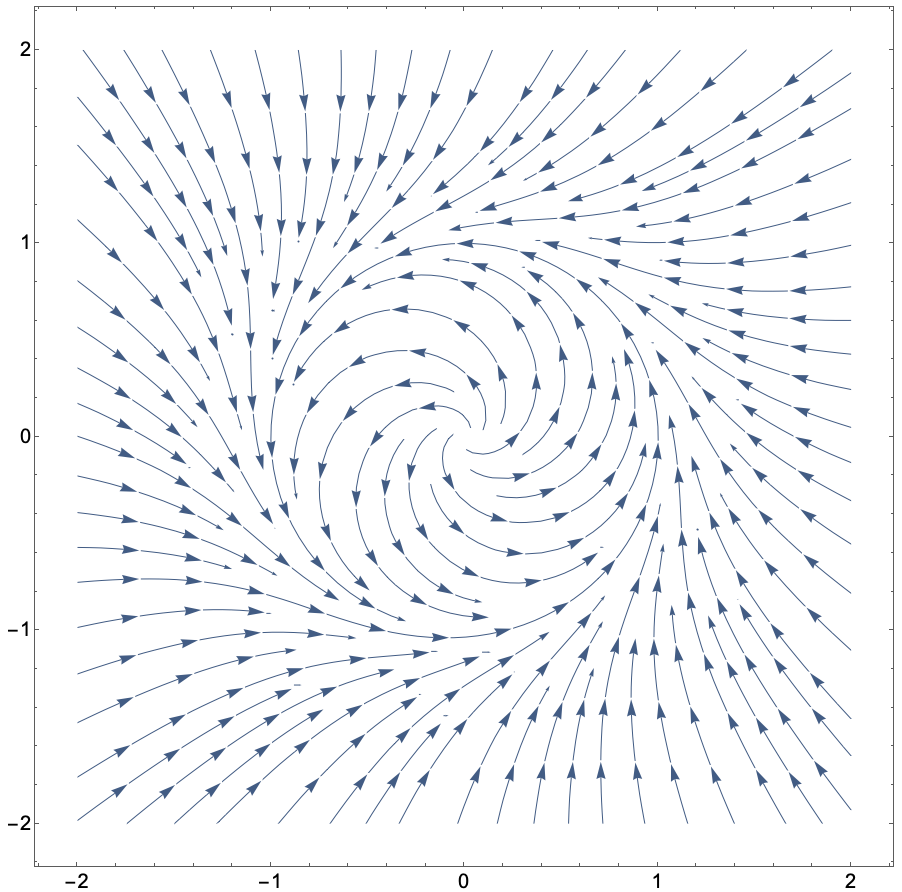}
\caption{{\color{black}Considering the vector field of the normal form of a Hopf bifurcation, Eq.~(\ref{hopfn}), for $\mu=1$ we see a limit cycle at radius $r=1$. A stream plot of several solutions of the flow are shown here and several eigenfunctions are shown in Fig.~\ref{shoot}.}}
\label{shoota}
\end{figure}

\begin{figure}[htbp]
\centering
\includegraphics[scale=0.22]{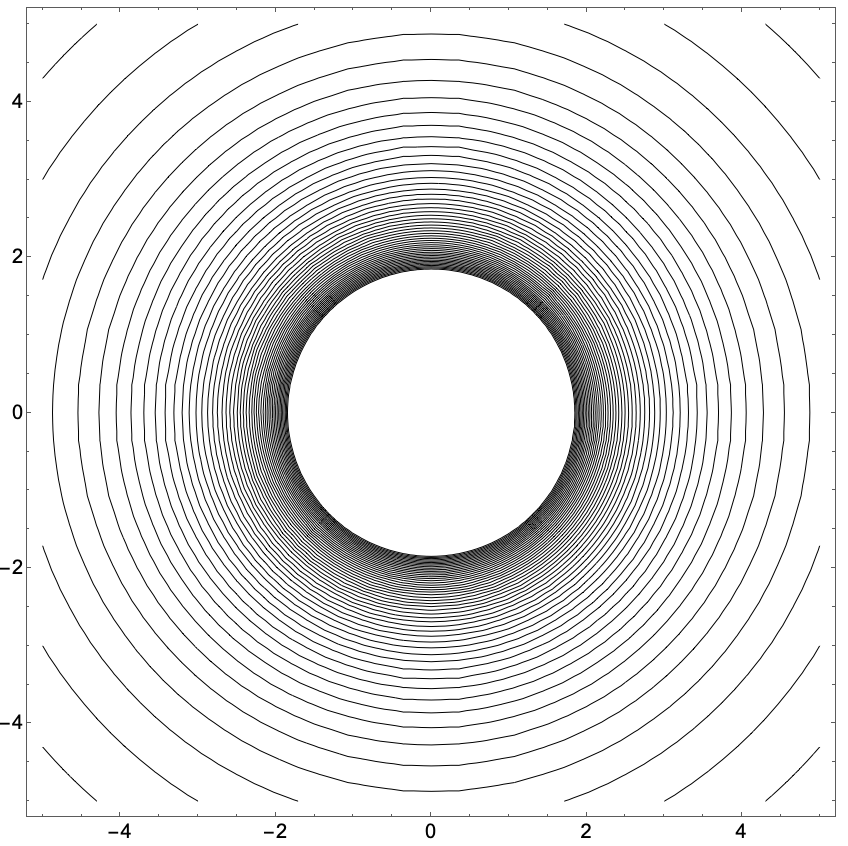}
\includegraphics[scale=0.215]{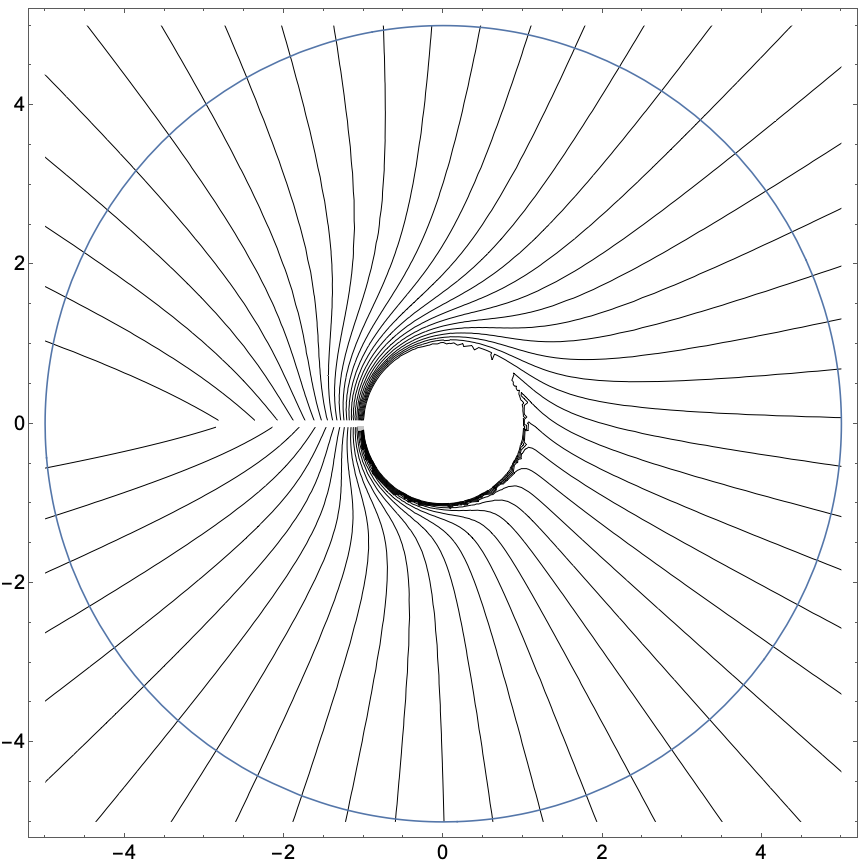}
\includegraphics[scale=0.22]{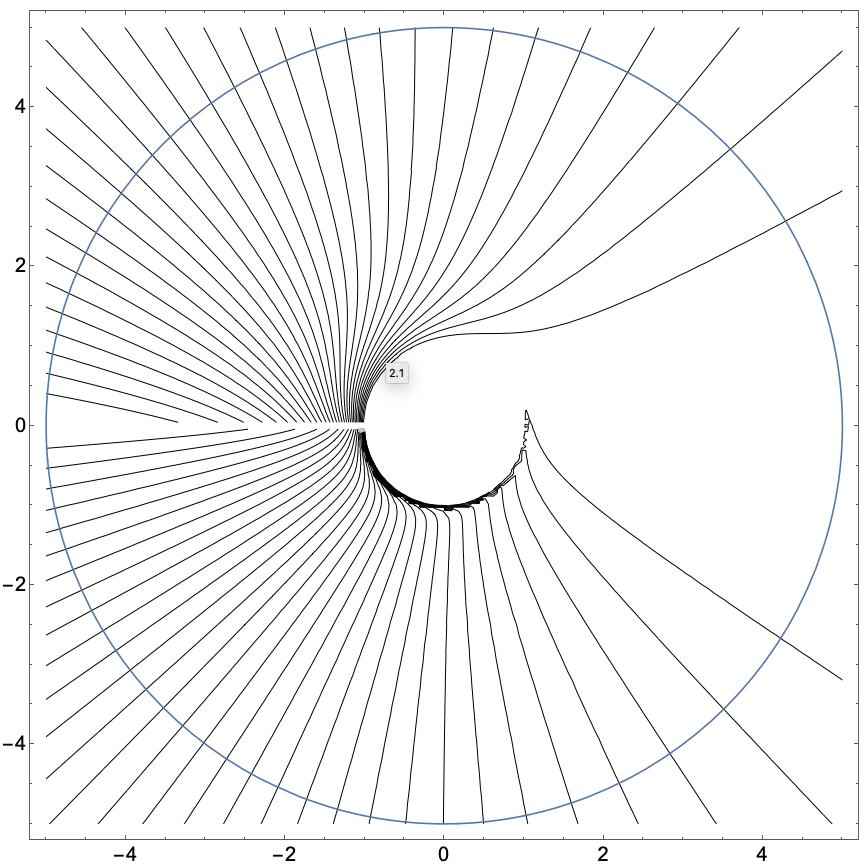}
\includegraphics[scale=0.222]{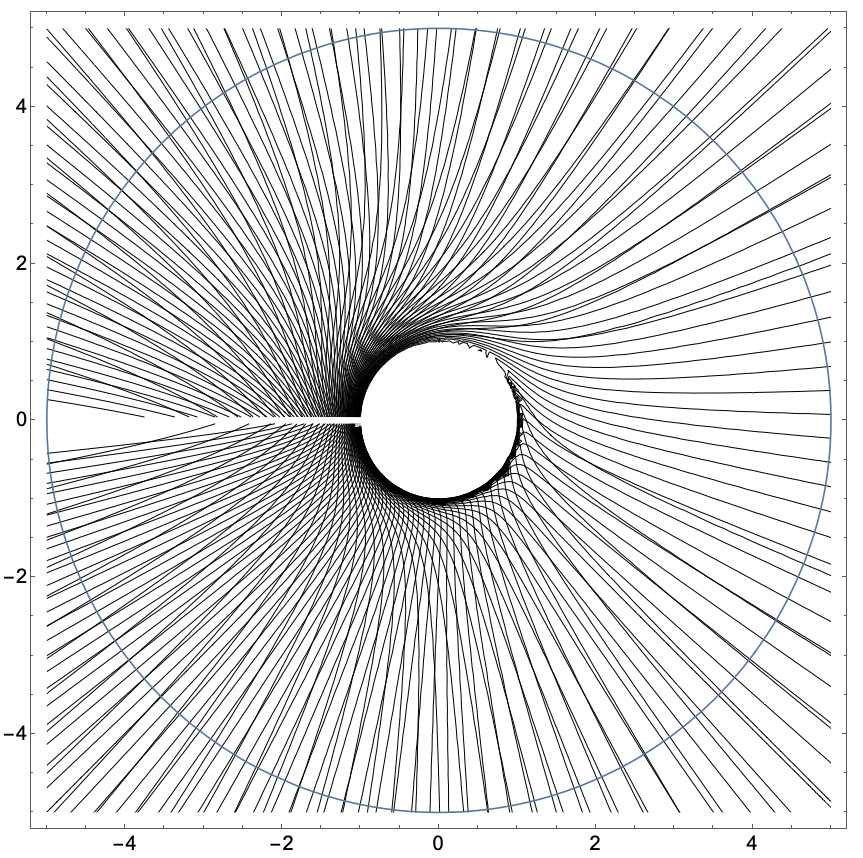}
\caption{{\color{black} Considering the vector field of the normal form of a Hopf bifurcation, Eq.~(\ref{hopfn}), and a stream plot is shown in Fig.~\ref{shoota}. (Upper left) A data curve $\Lambda$ is chosen to be a circle of radius $R=5>1$ (blue curve), and  an eigenfunction $\phi_\lambda$ follows, illustrated by its level sets, and here shown in an open domain $U$ that is the annulus, $1=r<\sqrt{x^2+y^2}$, using initial data $h(s)=1$. (Upper right) The same eigenvalue is chosen, but initial data on $\Lambda$ is defined $h(s)=s$. Notice the level sets are clearly transverse to those of the function shown upper left.  (Lower left) Level sets of $\phi_\lambda$ resulting from data $h(s)=s^2$ on the same circle $R>r$. (Lower right) notice that the level sets resulting from the eigenfunction $\phi_\lambda$ from $h(s)=s$ and $h(s)=s^2$ clearly differ as they cross transversally. This is an illustrative example for Theorem \ref{transversetheorem} and Eqs.~(\ref{funnyPDE})-(\ref{funnyPDE2}).}}
\label{shoot}
\end{figure}

\end{remark}

See Fig.~\ref{shoot}, where it is clear that the level sets of these eigenfunctions, even due to same eigenvalues,  but different initial data $h$, are not equivalent.  So by definition, likewise the primary KEIGS are not equivalent. {\color{black} Summarizing, eigenfunctions depend also on the data curve and also the specific data function on the data curve.  Our full notation is expanded to include these, $\phi_{\lambda, h,\Lambda}(x)$, noting that it is possible that likely there is no equivalence, $\phi_{\lambda, h,\Lambda}(x) \dot{\neq} \phi_{\lambda, \tilde{h},\Lambda}(x)$, even with the same eigenvalue $\lambda$ and even if the same data curve $\Lambda$, but differing data functions, $h\neq \tilde{h}$.  It is however true that some data $h$  on a given data curve $\Lambda$ pulls back to be stated elsewhere in the domain by  other data curves $\tilde{\Lambda}$ pulled back accordingly to a new $\tilde{h}$, already noted at Eq.~(\ref{pullbackcurve}). } 

The generality of this observation is straightforward as summarized by the following statement:
{\color{black}
\begin{theorem}\label{transversetheorem}
Given KEIGS with the same eigenvalue $\lambda$, but differing data functions $h,\tilde{h}$ on an otherwise same initial data set, $h,\tilde{h}:\Lambda\rightarrow {\mathbb C}$, we expand the KEIGS notation, $(\lambda, \phi_{\lambda,h,\Lambda}(x))$, and $(\lambda, \phi_{\lambda,\tilde{h},\Lambda}(x))$.  These KEIGS are generally not equivalent in the sense of Definition \ref{defnequiv}, $\phi_{\lambda, h,\Lambda}(x) \dot{\neq} \phi_{\lambda, \tilde{h},\Lambda}(x)$, unless the following geometric statement concerning compatibility of  the data functions $h$ and $\tilde{h}$, stated on $\Lambda$,
\begin{equation}
 \nabla \phi_{\lambda,h,\Lambda}\cdot \nabla_\perp \phi_{\lambda,\tilde{h},\Lambda}=0,
 \end{equation}
 holds.  In a two-dimensional phase space, 
 the data compatibility statement reduces  for  $y\in {\Lambda}\subset U$, to,
 \begin{equation}
 <h(y), \tilde{h}(y)>\cdot <\tilde{h}'(y), -h'(y)>=0.
 \end{equation}
\end{theorem}
}

Consider if $\nabla \phi_{\lambda,h}\cdot \nabla_\perp \phi_{\lambda,\tilde{h}}\neq 0$.
Referring to the general solution of a KEIGS for a  chosen initial data function on a data surface, $h(s)$ on $\Lambda$, then by Eq.~(\ref{gensol}), with notation extended to include these relevant information, 
$    \phi_{\lambda, h, \Lambda}(x)=h\circ s^*(x)e^{\lambda r^*(x)}$.  Likewise, for a different data function $\tilde{h}(s)$, but on the same data surface, $    \phi_{\lambda, \tilde{h}, \Lambda}(x)=h\circ s^*(x)e^{\lambda r^*(x)},  $ observe that for a given point $x$ these share the same point $s^*(x)$, and time of flight back along the flow to $\Lambda$, $r^*(x)$ (as defined in Theorem \ref{gensolkoop}). Then consider the orientation of level sets of $    \phi_{\lambda, h, \Lambda}(x)$ relative to those of $    \phi_{\lambda, \tilde{h}, \Lambda}(x)$ at a point $x=(x_1,x_2)$.  {\color{black} We specialize this discussion to two dimensions for simplicity, but the outcome is general and similar to show.} To decide if level sets of $\phi_{\lambda,h,\Lambda}$ are tangent to level sets of $\phi_{\lambda,\tilde{h},\Lambda}$ is described by,
\begin{eqnarray}\label{non1}
  & &  \nabla \phi_{\lambda,h,\Lambda}\cdot \nabla_\perp \phi_{\lambda,\tilde{h},\Lambda}=\lambda e^{2 \lambda r^*(x)}(-\tilde{h}\circ s^*(x) h'\circ s^*(x)+h\circ s^*(x) \tilde{h}'\circ s^*(x))(\frac{\partial s^*(x)}{\partial x_2} \frac{\partial r^*(x)}{\partial x_1}-\frac{\partial s^*(x)}{\partial x_1} \frac{\partial r^*(x)}{\partial x_2})=  \nonumber \\
   &=& \lambda e^{2 \lambda r^*(x)}(<h(y), \tilde{h}(y)>\cdot <\tilde{h}'(y), -h'(y)>)|_{y=s^*(x)}
    (<\frac{\partial r^*(x)}{\partial x_1}, \frac{\partial r^*(x)}{\partial x_2}>\cdot <\frac{\partial s^*(x)}{\partial x_2}, -\frac{\partial s^*(x)}{\partial x_1}>)
    = 0. 
\end{eqnarray}
Thus at $y=s^*(x)$, for $x\in U$ but $y\in {\Lambda}$,
 the following data compatibility statements relate to the question of tangency,
\begin{equation}\label{funnyPDE} 
<h(y), \tilde{h}(y)>\cdot <\tilde{h}'(y), -h'(y)>=0 \end{equation}
or the following statement describing how $\Lambda$ is oriented with respect to the flow,
\begin{equation}\label{funnyPDE2}
<\frac{\partial r^*(x)}{\partial x_1}, \frac{\partial r^*(x)}{\partial x_2}>\cdot <\frac{\partial s^*(x)}{\partial x_2}, -\frac{\partial s^*(x)}{\partial x_1}>=0,
\end{equation}
that should be prevented if $\Lambda$ is chosen to be transverse to the flow, then $\phi_{\lambda,h,\Lambda}$ and $\phi_{\lambda,\tilde{h},\Lambda}$ have tangential level sets, despite different data $h$ and $\tilde{h}$. These computations serve to prove Theorem \ref{transversetheorem}. Again, for intuition we refer to Figs.~\ref{Fig4}-\ref{shoot}.


{\color{black}
\begin{remark}\label{countmore}
We have already noted that the cardinality of KEIGS is at least uncountable due to the algebraic property, but now  we notice that even with the quotient, the cardinality of primary KEIGS may be greater than uncountable due to the cardinality of initial data functions $h$, and the possibility that these may produce distinct primary KEIGS.  Consequently, principal eigenfunctions in the sense of Mohr-Mezic \cite{mohr2016koopman} generate a principal algebra of eigenfunctions and each of these associates with a primary eigenfunction by our Definition \ref{defnequiv}. However, Theorem \ref{transversetheorem} implies that there are primary eigenfunctions that are not generated from principal eigenfunctions.
\end{remark}
}

In order to describe general observation functions as linear combinations of eigenfunctions, and given that primary eigenfunctions functions  summarize geometric aspects of eigenfunctions, we are motivated to identify pairwise primary KEIGS such that not only their level sets are transverse to each other (definitive  of sets of functions forming two different primary KEIGS) but ``decisively" so, as a pre-requisite to building efficient representation of general functions by eigenfunctions, as Definition \ref{efficientr}.   In the next section we will define an optimization based method in the spirit of POD to generate such functions.



\section{Numerical Methods: optimal Koopman Eigenfunction Extended Dynamic Mode Decomposition (oKEEDMD)}\label{numericalmethods}

We have described above how there are many eigenfunctions, generally at least uncountably many.  However,  many pairings of eigenfunctions present either the same, or almost the same set of level sets, that we define as an equivalence called primary KEIGS.   In \cite{korda2019optimal}, authors speak of having enough eigenfunctions for complete representation of arbitrary functions as linear combinations of eigenfunctions, following a phrase ``richness", as we also use a similar phrasing in \cite{li2017extended}.  But in our work we described how to develop an efficient basis set in terms of a machine learning based ``dictionary" concept that we called EDMD-DL, \cite{li2017extended}.  

In \cite{folkestad2019extended}, the authors have developed a variant of EDMD where the observation functions are themselves Koopman eigenfunctions, that they called
Koopman Eigenfunction Extended Dynamic Mode Decomposition (KEEDMD).
In broad sense this is what we do here, but with the extra technology so as to make for an optimally efficient representation.
Therefore in the spirit of naming DMD methods with acronyms, the method we present here might be called, oKEEDMD, the extra ``o" for optimal.  Also let us point out that our construction of eigenfunctions is very different than any other approach in that we restrict out construction {\color{black} to the co-dimension one manifold $\Lambda$,} to building a (an optimal) data function $h$ on a surface $\Lambda$ that is transverse to the flow.

Now in this paper we are positioned to approach a simple problem for practical applications, which is to ask how few eigenfunctions, optimally how few, can we use for a quality of estimate. Here we present a method to construct empirical eigenfunctions to solve a problem of efficient representation, and in some sense this problem reminds us of the celebrated POD-KL (Principal Orthogonal Mode - Karhunen-Loeve) concept that is popular in spatiotemporal data driven methods for constructing time-averaged optimal modes often from PDE data, \cite{karhunen1947under, loeve1977elementary, webber1997karhunen, holmes2012turbulence, watanabe1969knowing}.  There are distinct differences, most notably our representation is in terms of Koopman eigenfunctions.  Nonetheless, what is key in our construction is that while there may well be at least uncountably many eigenfunctions, we have {\it }freedom to choose those we like. So we assert an efficiency principle to choose good ones, and thus follows the title of the paper regarding a good dictionary.  See Definition \ref{efficientr}.  {\color{black}In representing a general observation by series of eigenfunctions, we are approximately decomposing it into standing modes, and by doing so efficiently, our goal is to use relatively fewer such modes if possible.}

In \cite{korda2019optimal}, the authors present a theorem  that asserts that an arbitrary continuous functions can be represented as linear combinations of eigenfunctions.  At the heart of the proof of their theorem is the idea (we now describe using our notation here) that eigenfunctions in a nonrecurrent domain $\Omega$ may be constructed on a transverse data surface $\Lambda\subset \Omega$, and a general (not necessarily eigen) function $q:\Omega \rightarrow {\mathbb C}$, that  pulls back to $\Lambda$ where careful choice of the data $h$ allows us optimal efficiency.

\begin{figure}[h]
\centering
\includegraphics[scale=1]{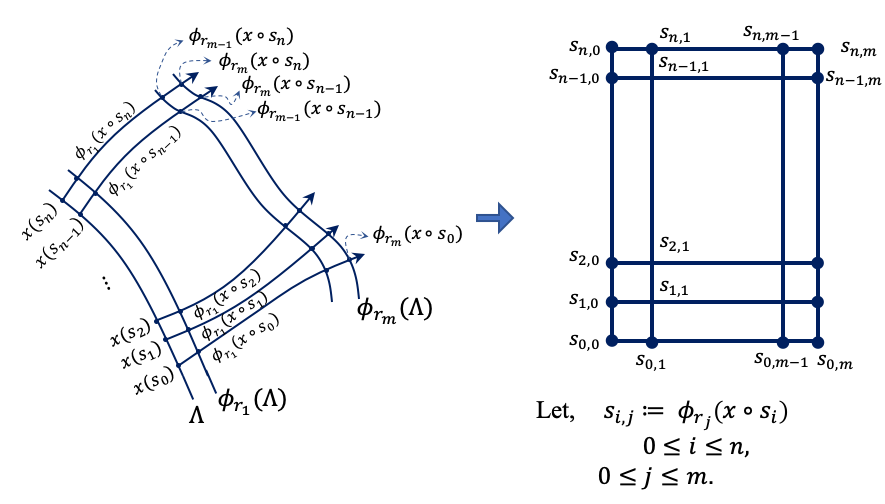}
\caption{Numerical optimization on a gridded sample of the optimal data function $h:\Lambda \rightarrow {\mathbb C}$ in terms of Eq.~(\ref{optvarphi}) proceeds on a grid, propagated by the flow to sweep out the domain of fitting, {\color{black} $U=\{\varrho_t (x), x\in \Lambda, t\in [t_1,t_2]\}$}, which by Eqs.~(\ref{grid1})-(\ref{grid2}) leads to a least squares problem for $h$ on an equal time grid (right), interpreted in the phase space (left).  Repeated application of the least squares problem Eq.~(\ref{opt1}) leads to a basis of successively optimal eigenfunctions, Eq.~(\ref{opt2}).}
\label{Figq}
\end{figure}

Now we will address efficiency of a Koopman eigenfunction representation.
 In the context of the geometric aspects described in earlier sections,  many pairs of eigenfunctions may not be significantly distinguished from each other in the sense that their level sets may either coincide or may be almost not transverse, so general representations may not be efficient.
So we pose the efficient optimal representation by eigenfunctions problem as follows.
Given an $L^2(U)$ function $q:U \rightarrow {\mathbb C}$, let,
\begin{equation}\label{optvarphi}
 \varphi=  arg\min_{\lambda, h} \ \|\phi_{\lambda, h}-q\|_{L^2(U)}^2
\end{equation}
be the eigenfunction $\varphi:U\rightarrow {\mathbb C}$ with eigenvalue $\lambda$ that most closely estimates $q\in L^2(U)$. The domain stated $U=\{\varrho_t (x), x\in \Lambda, t\in [t_1,t_2]\}$, and for presentation here, let $t_1=0$, but as long as $0 \in [t_1,t_2]$ the construction proceeds similarly.
So that we may refer to Theorem \ref{gensolkoop}, we will describe this problem in the Sobolov space $\phi_{\lambda, h}\in H^2(U)$.  Notice that the optimization is over eigenvalues $\lambda$ and data functions $h$, on some specific data curve $\Lambda$, so this is quite different from a typical eigenfunction optimization problem that may directly attempt to construct the eigenfunction $\phi_{\lambda, h}$.  The argmin will sufficiently produce   just a data function $h$ (and corresponding $\lambda$) and so that is co-dimension-one (lower dimensional) optimization problem as $\Lambda\in U$. 

Recall that this minimization problem coincides with maximal optimal projection as follows,
\begin{equation}
    arg\max_{\lambda, h} \ (q,\phi_{\lambda, h})= arg\min_{\lambda, h}  \ \|\phi_{\lambda, h}-q\|_{L^2(U)}^2,
\end{equation}
using norm and inner product notation as usual,
\begin{equation}
    \|f_1\|_{L^2(U)}=\int_U f_1(x) dx, \mbox{ and, }(f_1,f_2)=\int_U f_1(x) f_2(x) dx, \mbox{  for any $f_1,f_2\in L^2(U)$}.
\end{equation}
See Fig.~\ref{Figq}.

We proceed to interpret and approximate this optimization problem to develop optimal KEIGS (oKEIGS) based on  interpreting Theorem \ref{gensolkoop}, Eq.~(\ref{gensol}), beginning with a grid on the data surface $\Lambda$ propagating through the domain $U$.  This leads to a least squares problem as follows.  Let $s_0<s_1<...<s_n$ be a uniform partition of a data curve $\Lambda$. (In a multivariate setting a general grid will comparably lead to a linear inverse problem, but we proceed to describe the idea in this simpler case that $\Lambda$ is one-dimensional).  Data $h:\Lambda \rightarrow {\mathbb C}$ is likewise partitioned and indexed, 
\begin{equation}\label{grid1}
h_i:=h(s_i).
\end{equation}
Over a uniform  grid in time, $r_0=0<r_1<...<r_m$,
let,
\begin{equation}\label{thearray}
S_{i,j}:=\varrho_{r_j}\circ x(s_i), \mbox{ } 0\leq i \leq n, \mbox{ } 0\leq  j \leq m,
\end{equation}
in terms of the coordinate function into the domain $x(s):\Lambda \rightarrow U\subset M$.  This propagates by the flow $\varrho_r(\Lambda)$ for the gridded sample in space, and time, as shown in Fig.~\ref{Figq} indexed  uniformly on time  level of  $s^*(x)$ starting from $\Lambda$, (defined in Eq.~(\ref{sstar}) ), that are in turn level sets in $h$.  We have not discussed convergence of the discrete computational method here to the exact solution of the generator PDE, but it may be considered straightforward work of numerical analysis as we have restricted the Koopman operator dynamics to the composition operator for observation functions relative to Lebesgue measure,  $g\in L^2(M)$.

In terms of the grid, the optimization problem Eq.~(\ref{optvarphi}), for the function $\varphi(x)$ represented on the grid $\varphi \circ x(s_{i,j})$ is approximated by solving the finite rank least squares problem, 
\begin{equation}
     \tilde{\varphi}=arg\min_h \|\phi_{\lambda, h}\circ x(s_{i,j})-q\circ x(s_{i,j})\|_F^2=arg\min_{h_i} \sum_{j=1}^m \sum_{i=1}^n |e^{\lambda r_j}h_i-q_{i,j}|.
\end{equation}
The tilde, $``\mbox{ }\tilde{ }\mbox{ }"$ describes that the vector is an array representing the function on the grid $S_{i,j}$ at points Eq.~(\ref{thearray}) shown in Fig.~\ref{Figq}, and 
\begin{equation}
\tilde{\varphi}_{\lambda, h}\approx  \varphi_{\lambda, h}
\end{equation}
on that grid.
Stating both functions $q$ and $\tilde{\varphi}_{\lambda, h}$ are discretely sampled at 
\begin{equation}\label{grid2}
\varrho_{r_j}\circ x(s_i), \mbox{ } 0\leq i \leq n, 0\leq   m,  \mbox{ }
\mbox{ }q_{i,j}:=q(S_{i,j}).
\end{equation}
The usual Frobenius norm for arrays, $\|v\|_F^2=\sum_{i,j} v_{i,j}^2$ is used. The problem becomes more convenient when the arrays are reshaped.   To this end, we write the optimal initial data as an $n\times 1$ vector $h^o(\lambda)\in {\mathbb C}^n$ that solves a classical least squares problem,
\begin{equation}\label{opt1}
    h^o(\lambda)=arg\min_h \|A(\lambda) h- b\|_2,
\end{equation}
where,
\begin{equation}\label{opt2}
    A(\lambda)=E(\lambda) \otimes I_n, \mbox{ where, }b=reshape(q,mn,1),
\end{equation}
and $E(\lambda)$ is an $m\times 1$ vector, 
\begin{equation}\label{opt3}
E(\lambda)=[e^{\lambda r_0} e^{\lambda r_1} ... e^{\lambda r_m}]^t \in {\mathbb C}^m, 
\end{equation}
and $I_n$ is the $n\times n$ identity matrix.  With $\otimes$ as the Kronecker product, this makes $A(\lambda)$ an $mn \times n$ matrix consisting of the $E$ vector repeated $n$ times.  Further, $b$ is an $mn \times 1$ data vector describing the $n \times m$ array for the grid sample of the data function $q(S_{i,j})$.  The eigenfunction approximation, but reshaped as an $mn\times 1$ vector, follows the the least squares solution, $h^o(\lambda)$,
\begin{equation}
     p^o(\lambda)=A(\lambda)  h^o(\lambda),
\end{equation}
 to be reshaped to the domain, as an $m\times n$ array of estimated function values, $\tilde{\varphi}^o(\lambda)$.
Notice the arbitrary flexibility in choosing $\lambda$. As such we have written each of the expressions in Eqs.~(\ref{opt1})-(\ref{opt3}) as functions of $\lambda$.
  Any $\lambda$ will lead to an optimal vector $h^o(\lambda)$ and thus $\tilde{\varphi}^o(\lambda)$.  This reflects the same flexibility already noted, but now in other words.  In fact, for a given $q$.
  
Finally, we summarize,
\begin{eqnarray}\label{opt4}
    (\lambda^o_1,h^o_1) &=& arg\min_{\lambda, h} \|A(\lambda) h- b\|_2, \nonumber \\
    p_1^o&=&A(\lambda^o_1)  h^o_1, \nonumber \\
    \tilde{\varphi}_1^o&=& reshape(p_1^o,n,m).
\end{eqnarray}
Notice the subindex $``1"$ of $\tilde{\varphi}_1^o$ is placed to describe that further optimization may proceed as below to successively reduce the residuals.

\begin{figure}[htbp]
\centering
\includegraphics[scale=0.37]{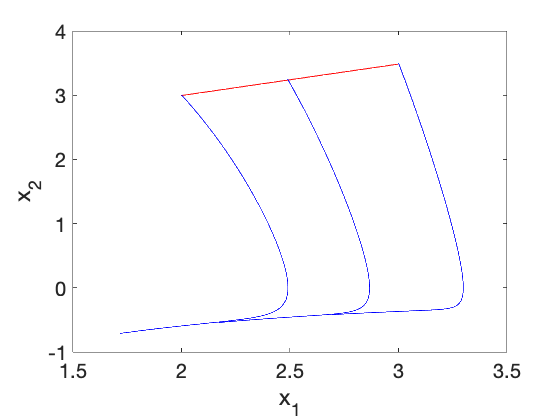}
\includegraphics[scale=0.37]{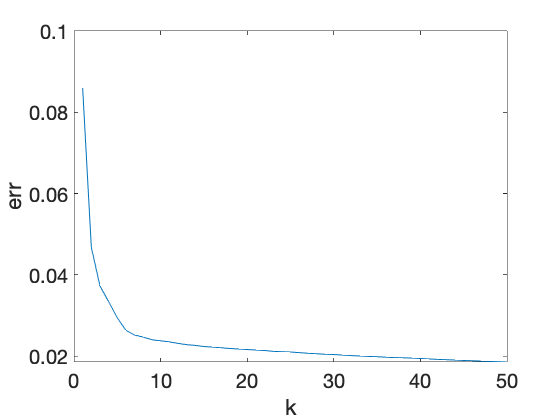}
\includegraphics[scale=0.24]{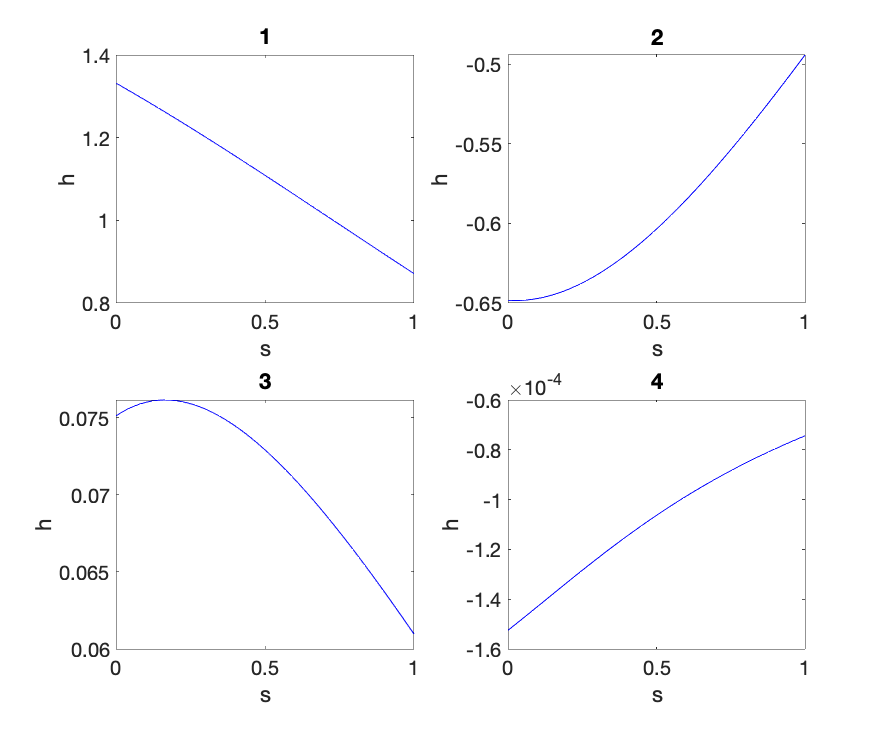}
\includegraphics[scale=0.23]{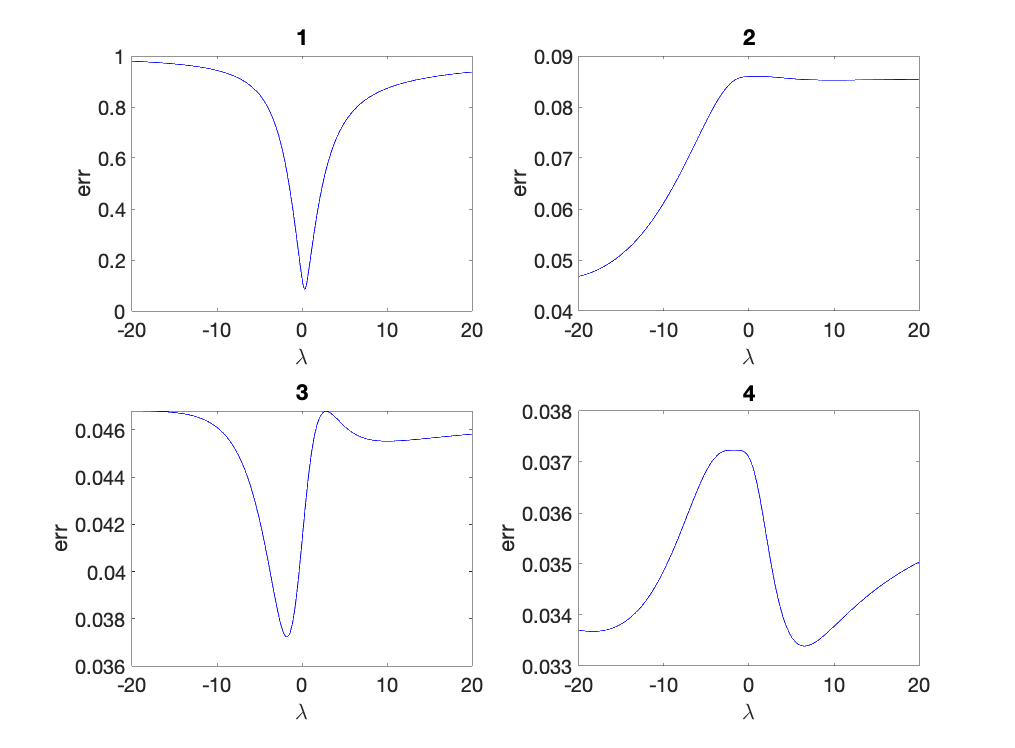}
\caption{oKEEDMD construction for the Van der Pol equation, Eq.~(\ref{vander}). (Upper Left) {\color{black} The red data curve $\Lambda$ 
evolves by the flow to sweep out domain $U$ shown between the blue curves, (and a blue curve through the middle is shown for reference).  All subsequent figures consider the specific target observable function Eq.~(\ref{thisdata}), in this domain.  } (Upper Right) Optimal residual as a function of number of terms fitted, $k$.  {\color{black} (Lower Left)  {\color{black} The first $4$ optimally fitted data functions $h$ on the (red) manifold $\Lambda$ for target function $q(x)$ in the domain shown. }}  (Lower Right) The range of errors when fitting the function, and its residuals, shown as a function of $\lambda$, the optimal eigenvalue successively being the minimal value on each curve.  }
\label{Figq2}
\end{figure}

Continuing in the theme of estimation, we can successively fit the residuals,
\begin{eqnarray}\label{opt2}
    R_k&=&b- \sum_{l=1}^{k} p^o_l, \nonumber \\
  (\lambda^o_{k+1},h^o_{k+1}) &=& arg\min_{\lambda, h} \|A(\lambda) h- R_k\|_2, \nonumber \\
    \tilde{\varphi}_{k+1}^o&=& reshape(p_k^o,n,m).
  \end{eqnarray}
It is convenient to normalize each of these eigenfunctions,
\begin{equation}
        \overline{\varphi}_k^o=\frac{\tilde{\varphi}_k^o}{c_k},       c_k=\|\tilde{\varphi}_k^o\|_2,
\end{equation}
but there is no reason to assume these form an orthogonal set of eigenfunctions.
Then follows the statenent of optimal estimation by these empirical oKEIGS, 
\begin{equation}\label{opt222}
    q=\sum_{l=1}^k c_k \overline{\varphi}_k^o + R_{k+1}.
\end{equation}
Coefficients $c_k$ describe the importance of the dynamic pattern of each of these oKEIGS.  {\color{black} At each stage, fitting the residual, we successively search for best eigenfunctions, and each eigenfunctions  is independently associated with a data function on $\Lambda$.}

 \paragraph{{\color{black}Example: oKEEDMD for the Van der Pol oscillator.}}
See Fig.~\ref{Figq2} showing example optimal solutions of a specific example problem, {\color{black} and specifically including the resulting data functions $h$.} Let,
\begin{equation}\label{thisdata}
q(x)=3e^{\frac{-(x_1^2+x_2^2)}{10}},
\end{equation}
$q:{\mathbb R}^2\rightarrow {\mathbb R}$, chosen to challenge the method since it is not an eigenfunction, and also for an interesting shape.  We will construct an approximation of this general function in terms of eigenfunctions from a Van der Pol equation, \cite{van1926lxxxviii}, 
\begin{equation}\label{vander}
\dot{x}_1=x_2, \dot{x}_2=x_2(1-x_1^2)-x_1,
\end{equation}
 akin to the the Hopf system already studied above, Eq.~(\ref{hopfn}).  In Fig.~\ref{Figq2}a we show a line segment in red, that denotes a segment of an initial data curve $\Lambda$, and over a time interval $0\leq t\leq 2$ this curve sweeps out the fitting domain $U$ shown outlined in blue, and we also showed a curve through the middle for reference, $U\subset M={\mathbb R}^2$.
The specific domain, as well as the specific function $q$ to be fitted, and the specific dynamical system are all relevant to the specific optimal solution.

Continuing with the example, see Fig.~\ref{Figq2}b, showing the successively optimally decreasing residual error.  We see a distinct ``elbow" which is a commonly used criterion in reduced order modelling.  So we suggest here, that using just $4$ terms performs quite well.  Figs.~\ref{Figq2}c-d show the resulting fitted $h_i^o:\Lambda\rightarrow {\mathbb R}$ functions which then as above defines the eigenfunctions $\varphi_i^o:U\subset M\rightarrow {\mathbb R}$.  {\color{black} Each best $h$ at each stage is different,  and likewise each $\lambda$, successively considering the remaining residual $R_k$ Eqs.~(\ref{opt2}), (\ref{opt222}) through stages of the algorithm.}
Finally in Figs.~\ref{Figq2}d, we show a sweep of real valued eigenvalue candidates for $\lambda_i$, $i=1, 2, 3, 4$.  The best fitted eigenvalue is chosen and correspondingly the $h_i^0$ and $\varphi_i^o$.
By design, $\|R_k\|\downarrow $ faster for these KEIGS than when using any of the other of the infinitely many available eigenfunctions, and the hope is that $\|R_k\|\rightarrow 0$ as $k\rightarrow \infty$.   {\color{black} Compared to fitting an observable by other eigenfunctions, including as the KEEDMD \cite{folkestad2019extended}, we interpret in Fig.~\ref{Figq2}(Lower Right) that there can remain significant variability of quality of fit, that our oKEEDMD exploits.}

\section{Conclusion}

We have described  how the Koopman operator of a flow in an nonrecurrent domain leads to infinitely many eigenfunctions, for each eigenvalue, and that all eigenvalues are legitimate.  This means that the point spectrum is the full complex plane, but interestingly the spectral decomposition into point and complex spectrum in fact rely on the phase space domain.  Furthermore, there  at least an uncountable multiplicity of eigenfunctions for each eigenvalue.  This all follows the idea that arbitrary functions can be formed on a transverse to the flow data set (curve), and the method of characteristics forms the eigenfunction in the rest of the domain.  Even with a simple equivalence class notion that we define here to cope with the algebraic structure of eigenfunctions, whereby we define equivalence of those functions with matched sets of level sets, stating the quotient as a ``primary eigenfunction", still there are uncountably many primary eigenfunctions, thus even despite the algebraic structure of eigenfunctions.  Then we showed that there are simple geometric considerations to  differentiate primary eigenfunctions, when the goal is representation of general functions in terms of linear combinations of primary functions.  Thus follows a new kind of empirical scheme to produce efficient functions, that we call a good dictionary.  In the spirit of naming algorithms for empirical spectral analysis of the Koopman operator as DMD-like methods, here we call ours oKEEDMD, for ``optimal KEEDMD" where EDMD means empirical dynamic mode decomposition for the idea of using basis functions and KE is the new prefix from \cite{folkestad2019extended} for Koopman eigenfunction, meaning those basis functions are themselves chosen to be  eigenfunctions. {\color{black} In summary in main points, we 1) illustrate that eigenfunctions follow a quasi-linear PDE, which is explicitly solvable by the classical method of characteristics, that leads to an understanding of a cardinality of eigenfunctions, for each eigenvalue, that is greater than uncountable. 2) We introduce the concept of primary eigenfunctions as an equivalence class amongst all those functions whose set of level sets match, and this accounts for the principle algebraic structure, and more. 3) Still there are geometric distinguishing features that allow for infinitely many primary eigenfunctions. 4) Then, exploiting the richness available with so many primary eigenfunctions, we present a method called oKEEDMD to efficiently construct series of eigenfunctions to empirically fit a given observation.}

Our construction of a good dictionary is unique from all other constructions since respecting that all eigenfunctions of a nonrecurrent domain are definable entirely on a co-dimension one subdomain that is transverse to the flow, then it is a matter of picking a good data function.  So with this in mind we assert a notion of efficient representation that relies just on that basis function but whose argument is only the data function defined only on the co-dimension one set.  Thus follows oKEEDMD.  
We hope that this idea of building eigenbasis functions for reduced order modelling may well find applications beyond the Koopman operator setting where this work was formulated.

\section{Acknowledgements}
The author received funding from the Army Research Office (N68164-EG) and also DARPA.  The author thanks Felix Dietrich for helpful feedback, and also the very thorough and insightful comments by the anonymous referees.

\section{Appendix: The Method of Characteristics for the Koopman Infinitesimal Generator PDE}\label{characteristics}

Here we recall the classic method of characteristics  for quasi-linear PDEs and and this is important to us here since this   includes the Koopman PDE, Eq.~(\ref{kooppde}) as a member of this class. For instructive practice, will also specialize the approach to solve the Eq.~(\ref{kooppde}) in an especially simple scenario of a linear ODE.
We closely follow the description and notation from the well-regarded textbook by F.~John, \cite{john1975partial}.

A general first order PDE of $u(x,y,...,z)$ may have a general form, $f(x,y,...,u,u_x,u_y,..,u_z)=0,$ but if $f$ can be re-arranged algebraically so that the PDE can be written as, 
\begin{equation}\label{pdechar}
    a(x,y,u)u_x + b(x,y,u) u_y=c(x,y,u),\mbox{ } (x,y)\in \Omega
\end{equation}
then it is called a quasi-linear PDE, and it is amenable to the following solution method.  Our presentation is specialized to two spatial variables, but the concept is more general and similar.  Allowing a solution surface, $z=u(x,y)$ in an extended $(x,y,z)$ space, then the functions $a, b, c$ may be thought of as a vector field. For brevity, we are skipping the extended discussion of the function domains of enough regularity to carry forward the following. Thus there corresponds a function $U$, such that $z=u(x,y)\iff U(x,y,z)=0$, and $\nabla U=<u_x,u_y,-1>$ is normal to the surface, and the PDE Eq.~(\ref{pdechar}), can be written $\nabla U \cdot <a,b,-c>=<u_x,u_y,-1>\cdot<a,b,-c>=0$.  Therefore since perpendicular's to the surface ($\nabla U$) are also perpendicular to the vector field $<a,b,c>$, then $<a,b,c>$ is a vector field that is tangent to the {\it characteristic curves.}  If we parametrize this surface by $r$, then we write an ODE for the curves,
\begin{eqnarray}\label{odeeq}
\frac{dx}{dr}=a(x,y,z), \nonumber \\
\frac{dy}{dr}=b(x,y,z), \nonumber \\
\frac{dz}{dr}=c(x,y,z).
\end{eqnarray}  
We assume  enough regularity for existence and uniqueness, such as the widely used assumption $<a,b,c>\in Lip(M)$, \cite{perko2013differential}. 

Taking Eq.~(\ref{pdechar}) as a Cauchy problem stated with an initial profile on a curve,  
\begin{equation}
    \Lambda=\{(x,y,z)=(f_1(s),f_2(s),h(s))\},
\end{equation}
this reduces a solution to the PDE Eq.~(\ref{pdechar}) as the data function, $h$ on $\Lambda$, as,
\begin{equation}
    h(s)=u(f_1(s),f_2(s)).
\end{equation}
This curve $\Lambda$ should be transverse to the characteristic curves, stated,
\begin{equation}\label{transverse}
    \frac{d}{ds}(f_1(s),f_2(s),h(s))\cdot <a(f_1(s),f_2(s),h(s)),b(f_1(s),f_2(s),h(s)),c(f_1(s),f_2(s),h(s))>\neq 0, \mbox{ for all } s.
\end{equation} 
Correspondingly, the ODE statement on characteristics, Eq.~(\ref{odeeq}) is stated as an initial value problem for each $s$.

\begin{remark}{\bf Example: The Observer Eigenfunction of a Linear System.}\label{linearexample}
To demonstrate the method of characteristics for Koopman eigenfunction analysis, first we recall a now well known scenario, which is a linear ode, that produces the so-called ``observer" function as an eigenfunction. 
Consider again the Koopman PDE, Eq.~(\ref{kooppde}), here specifically for a linear flow.  For simplicity of presentation, choose $d=2$ dimensional domain $M$, and choose the underlying ODE Eq.~(\ref{ode}) to be,
\begin{eqnarray}\label{simplestlinear} 
\dot{x_1}=a_1 x_1 \nonumber \\
\dot{x_2}=a_2 x_2.
\end{eqnarray}
Note that otherwise, the  subsequent analysis is largely the same for higher dimensional linear problems, $d>2$.
Then the Koopman PDE is,
\begin{equation}\label{l1}
    a_1 \phi_{\lambda,x_1}+a_2 \phi_{\lambda,x_2}=\lambda \phi_\lambda
\end{equation}
for $\phi_\lambda(x_1,x_2)$ and we have overloaded the notation, that $\phi_{\lambda,x_i}\equiv \frac{\partial \phi_\lambda}{\partial x_i}$.
For specificity we assume that the transverse data curve $\Lambda$ has a specific simple form, 
\begin{equation}\label{specialdata}
    \Lambda=\{(x,y,z)=(s,1,h(s))\},
\end{equation}
which is a horizontal line.  

Then the characteristic curves, Eq.~(\ref{odeeq}) follow,
\begin{eqnarray}
    \frac{dx}{dr}&=&a_1 x, \nonumber \\
    \frac{dy}{dr}&=&a_2 y, \nonumber \\
    \frac{dz}{dr}&=&\lambda z.
\end{eqnarray}
As an initial value problem on $\Lambda$ with data function $h(s)$, we have a  solution,
\begin{equation}
    x=s e^{a_1 r}, y=e^{a_2 r}, z=e^{\lambda r}h(s),
\end{equation}
by multiplying factors.  Eliminating the variable $r$ (parameterization along a characteristic curve, starting at any given parameterization point $s$ on $\Lambda$), 
$\ln y=a_2 r \Rightarrow r=\frac{\ln y}{a_2} \Rightarrow e^{a_1 r}=e^{\frac{a_1}{a_2} \ln y}=y^{\frac{a_1}{a_2}},$
and substituting the original variables, $(x,y,z)=(x_1,x_2,\phi_\lambda(x_1,x_2))$ yields,
\begin{equation}\label{ceigs}
    \phi_\lambda (x_1,x_2)=x_2^{\frac{\lambda}{a_2}}h\left(\frac{x_1}{x_2^{(\frac{a_1}{a_2})}}\right).
\end{equation}

A special simple  case follows the choices of 1) initial data is a constant, $h(s)=1$, on 2) the initial data curve, $\Lambda$ chosen specifically as a ``horizontal line", Eq.~(\ref{specialdata}).  Then we get what has been called the ``{\bf observer function}," when $\lambda_2=a_2$ and thus KEIGS,
\begin{equation}\label{obs1}
    (\lambda_2, \phi_{\lambda_2} (x_1,x_2))=(a_2,x_2).
\end{equation}
This is easily confirmed to be a solution by substitution into the Cauchy problem, Eqs.~(\ref{l1})-(\ref{specialdata}).  An important point is that to get the other ``typically" stated eigenfunction, the observer function in the other variable, we must choose a different data curve, and {\it this time not only must it be transverse to the flow, but it must also be transverse to the data curve,} Eq.~(\ref{specialdata}).  Similarly, this time choosing specifically a vertical line, 
\begin{equation}
\Lambda=\{(x,y,z)=(1,s,h(s))\},
\end{equation}
yields a general solution for general initial function $h$, 
\begin{equation}\label{ceigs2}
\phi_\lambda (x_1,x_2)=x_1^{\frac{\lambda}{a_1}}h\left(\frac{x_2}{x_1^{(\frac{a_2}{a_1})}}\right).
\end{equation}
Likewise this yields the observer function of a constant function $h(s)=1$, and $\lambda_1=a_1$, on the this time vertical data curve, 
\begin{equation}\label{obs2}
(\lambda_1, \phi_{\lambda_1} (x_1,x_2))=(a_1,x_1).
\end{equation}
\end{remark}
If the system is not already diagonal, standard linear theory to diagonalize the system by a similarity transformation of eigenvectors reveals the straightforward and similar results, remembering that conjugate systems share eigenvalues and corresponding eigenfunctions through the homeomorphism as change of variables, \cite{mezic2019spectrum,bollt2018matching}.

\begin{remark}{\bf Example: The Observer Eigenfunction of a One-Dimensional Linear Equation.}
An even simpler case follows the $d=1$-dimensional ode,
\begin{equation}
    \dot{x}=a x,
\end{equation}
where from the Koopman-``pde" is actually the following ode,
\begin{equation}
    a x \frac{d \phi_\lambda}{dx}(x)=\lambda \phi_\lambda(x)=\lambda \phi_\lambda(x).
\end{equation}
In 1-dimension we do not even need to resort to the method of characteristics.
The solution follows by the standard method of ODEs called multiplying factors, using initial data $\phi_0$,
\begin{equation}
    \phi_\lambda(x)=\phi_0 x^{\frac{\lambda}{a}}.
\end{equation}
Notice there has not as much freedom in one-dimension, other than algebraic powers, there is only a constant distinguishing eigenfunctions, but not a complete function of arbitrary data, and eigenfunctions are usually considered as an equivalence class when they simply differ by a constant factor.
\end{remark}

\begin{remark}
Notice that not only is there a solution eigenfunction Eq.~(\ref{ceigs}), one for each $\lambda \in {\mathbb C}$, but for each $\lambda$, there is also the freedom to choose the initial data function $h:\Lambda \rightarrow {\mathbb C}$ arbitrarily.  This is the source of a great deal of nonuniqueness that is only in part equivalent to the usual discussion of uniqueness of the Koopman spectrum associated with the algebraic property of eigenfunctions.
\end{remark}

\begin{proof} of Theorem \ref{thekoopmanpdetheorem}
Here we state the problem for a $d$-dimensional autonomous ODE.
Nonautonomous systems in $d$-dimensions can be similarly handled, first by the standard methodology of augmenting with a new variable to represent time, to $d+1$-dimensional autonomous system.
Restating that the initial value problem, $\dot{x}=F(x)$, $x(t_0)=x_0$ gives a unique flow $x(t)=\varrho_t(x_0)$ solution for $x_0\in U$, and $x\in U\in M$ is an open set containing the initial value curve, $\Gamma\subset U$.  Then in $U$, $\nabla \phi_\lambda \cdot F=\lambda \phi_\lambda$, has characteristic curves which are solutions of the ODE,
\begin{eqnarray}\label{pf1}
    \frac{dx}{dr}&=&F(x) \nonumber \\
    \frac{dz}{dr}&=& \lambda z,
\end{eqnarray}
following closely the formulation of the method of characteristics as above.
So then,
\begin{eqnarray}
x(r)&=& \varrho_r(x_0(s)) \nonumber \\
z(r)&=& e^{\lambda r}h(s).
\end{eqnarray}
Here we remind that for initial conditions $x_0=x_0(s)$ on the data surface $\Lambda\in U$, and since we write $s$ as the parameterization on $\Lambda$, then the initial value is,
\begin{equation}
\phi_\lambda(x_0)=\phi_\lambda(x_0(s))=h(s).
\end{equation}
Therefore, for an arbitrary $x\in U$, (and $x$ is not necessarily on the initial data curve $\Lambda$), by assumed existence and uniqueness of the flow, $\varrho_t:U\times {\mathbb R}\rightarrow U$, there exists an $r^*(x)\in {\mathbb R}$ such that,
\begin{equation}
    \varrho_{-r^*(x)}(x)\in \Lambda,
\end{equation}
uniquely by assumption of no recurrence and existence by assumption that $\Lambda$ partitions $U$, transversally to the flow, outcome of the assumption, $U=\cup_{t\in [t_1,t_2]} \varrho_t(\Lambda)$.  

Further, for since each $x\in \Lambda$ has a parametrization, $s(x)$ then let,
\begin{equation}
    s^*(x)=s\circ \varrho_{-r^*(x)}(x).
\end{equation}
Hence,
\begin{equation}\label{pf2}
    \phi_\lambda(x)=h\circ s^*(x) e^{\lambda r^*(x)}.
\end{equation}
That is, from $x\in U$, ``pull"-back along the flow from $x$ to $s^*(x)\in \Lambda$, where initial data $h$ ``is read", and the value of $\phi_\lambda(x)$ at $x$ is therefore $h\circ s^*(x)$, but scaled linearly by the time of the pullback, forward from $s^*(x)\in \Lambda$ back to the original point $x$: $e^{\lambda {r^*(x)}}$.  See Fig.~\ref{pullback}. \end{proof}
It is straightforward check to confirm that the example, Eq.~(\ref{ceigs}), that we derived by direct computation specialized to the linear problem above, can also be derived  by application of the general formula in Theorem \ref{gensolkoop}, Eqs.~(\ref{gensol}), (\ref{pf2}).


\section{Appendix: Example, On Point Spectrum and Continuous Spectrum}\label{ctsspect}

Much has been said lately about the importance of including or understanding  the continuous spectrum of Koopman operators with many interesting directions including for computation and in terms of the associated measures, \cite{mezic2019spectrum, korda2018data}, including for toral automorphisms, \cite{govindarajan2019approximation}.  Yet as far as we know, there is not yet been a study regarding the Koopman operator spectrum, at least in the modern data driven oriented literature, that appeals directly to the classical definition of the spectral theory of operators regarding the decomposition into point, continuous, and residual spectrum, as we reviewed Definition \ref{defn2} above.  In so doing the conclusion of this section is to contrast the significantly different results that follow different chosen domains. 

Here we will discuss aspects of the spectrum,
\begin{equation}\label{spectraldecomp}
\sigma({\cal L})=P_\sigma({\cal L})\cup C_\sigma({\cal L})\cup R_\sigma({\cal L}),
\end{equation}
of the Koopman infinitesimal operator ${\cal L}$, Eq.~(\ref{ininitesimal}), as well as of the Koopman operator $K_t$, Eq.~(\ref{koopmandefn}), Definition \ref{koopmandefnD}, and we will do so directly appealing to Definition \ref{defn2}.
As it turns out,  the specific choice of definition of the domain $U\in M, $ of the measurable functions $g:U\rightarrow {\mathbb C}$  involved is crucial to distinguish the nature of the spectral decomposition.

We will appeal to an
 interesting example which was highlighted recently in the work of Mezic, \cite{mezic2019spectrum}, which describes a perfect action-angle system and from this example we will draw our contrasting observations.  First we review: let $(I,\theta)\in M={\cal I}\times S^1$, where ${\cal I}$ is an annulus in the plane, ${\cal I}=[a,b]\in {\mathbb R}^+$ is an interval and $S^1$ is a circle, and so $M$ is a cylinder or alternatively it can be thought of as an annulus as shown in Fig.~\ref{wedgefig}.  As we will see, depending on the domain chosen, there are significant difference, and these are crucially rooted in the concept of nonrecurrence that featured centrally in Theorem \ref{gensolkoop}.

 Let, 
 \begin{equation}\label{actionangleode}
     \dot{I}=0, \mbox{ } \dot{\theta}=I.
 \end{equation}
 From an initial condition, $(I_0, \theta_0)$ clearly the solution is $(I(t),\theta(t))=(I_0,\theta_0+t I_0)$.  We will use this same problem to emphasize the role of domain in contrasting the case of when there is a point spectrum, and/or a continuous spectrum, and also how the cardinality of these belies their names.  
 
 \subsection{Nonrecurrent subdomain}
 
 In a nonrecurrent subdomain $U$ of $M$, Theorem \ref{gensolkoop} holds without modification, and therefore solutions of the KEIGS are easily determined and of the form,  Eq.~(\ref{gensol}).  Such a subdomain may be chosen as any open set that does not circuit a complete circle, e.g. an open ``wedge" $U=(a,b)\times (\alpha_1,\alpha_2) \subsetneq {\cal I}\times S_1$, $0<\alpha_1<\alpha_2<2 \pi$.  Consider for example,
 \begin{equation}
\Lambda=\{(s,\alpha): \alpha_1<\alpha<\alpha_2 \mbox{ is a fixed constant }, a<s<b \mbox{ is a parameter}\},
 \end{equation}is a radial ray in $U\subseteq {\cal I}$. See Fig.~\ref{wedgefig}.  Further assume a data function is given, $h:\Lambda \rightarrow {\mathbb C}$.  Then by Eq.~(\ref{gensol}),
 \begin{equation}
     \phi_{\lambda,h,\Lambda}(I,\theta)=h(I)e^{ \frac{\lambda( \theta - \alpha)}{I}}.
 \end{equation}
 Furthermore, by arguments also leading to Corollary \ref{gensolkoop}  and also Remark \ref{countmore} we know that all $\lambda \in {\mathbb C} $ are eigenvalues (and furthermore all $h$ allow eigenfunctions).  Therefore, in terms of spectral decomposition, we conclude that the point spectrum is the full complex plane, 
 \begin{equation}
P_\sigma({\cal L})={\mathbb C}, \mbox{ but therefore by compliment, the  continuous and residual spectra, }C_\sigma({\cal L})=R_\sigma({\cal L})=\emptyset,
\end{equation}
 completes the spectrum, $\sigma({\cal L})$ as per Eq.~(\ref{spectraldecomp}).
 
\begin{figure}[htbp]
\centering
\includegraphics[scale=1]{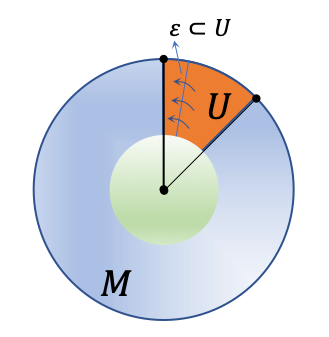}
\caption{ Domain in the annulus $(I\theta)\in M={\cal I}\times S^1$ be chosen as the entire annulus but then the nonrecurrence condition is violated and so Theorem \ref{gensolkoop} does not apply, but however if a wedge  $U=(a,b)\times (\alpha_1,\alpha_2) \subsetneq {\cal I}\times S_1$, $0<\alpha_1<\alpha_2<2 \pi$ is chosen, then the theorem does hold, and therefore a different form of the eigenfunctions hold with significant consequences to the nature of the spectral decomposition.
\label{wedgefig}
}
\end{figure}

 \subsection{Recurrent subdomain}
 
 Now consider the same dynamical system, the action angle flow Eq.~(\ref{actionangleode}), we reconsider the problem, but without the restriction on nonrecurrence.  So now we take $U=M={\cal I}\times S^1$ as the full annulus, as did Mezic in \cite{mezic2019spectrum}, so first we review.  With nonrecurrence, Theorem \ref{gensolkoop} does not hold, and therefore we may not resort to the eigenfunction solution Eq.~(\ref{gensol}) that follows from method of characteristics; that proof breaks down largely because of loss of the concept of a unique solution must be single valued.  However, Mezic  showed directly that since the Koopman operator $K_t[g](I,\theta)=g(I,\theta+It)$, he proved that eigenfunctions do not exist, stated  using the phrase ``the eigenfunction in a proper sense do not exist."  The argument was, depending directly on the idea of a composition operator rather than the infinitesimal generator,  that an eigenfunction would need to satisfy $\phi_\omega(I,\theta)=\phi_\omega(I+\theta+It)=e^{i\omega t}\phi_\omega(I,\theta)$, but allowing constant radius $r=I$ for this system and writing, $\phi_\omega=r^{i \varphi(I,\theta)}$ then also $\phi_\omega(I+\theta+It)=r e^{i\varphi(I,\theta)+iIt}=r e^{i\varphi(I,\theta)+iwt}$, which can only hold if $\omega=I$. The conclusion  is that this  not possible in the sense of a function, so there is no eigenfunction.  But the author goes on to note that,
 \begin{equation}\label{delta}
     \phi(I,\theta)=e^{i\theta}\delta(I-c),
 \end{equation}
where $\delta(x)$ is the dirac delta-``function", acts like an ``eigenfunction" but that is really only true in the distributional sense and not literally as a function.  Therefore he notes that this as an ``eigenmeasure,"  but it is not a function.  That is, weakly one demands that integrating against test functions, $f(I,\theta)$, the composition operator acts as, 
$\int_M K_t \phi(I,\theta)f(I,\theta) dI d\theta=\int_M  \phi(I,\theta+It )f(I,\theta+It) d I d\theta = \int_M  e^{i(\theta + I t)} \delta(I-c)f(I,\theta) dI d\theta= e^{i c t} \int_M \phi(I,\theta) f(I,\theta) dI d\theta.$  In fact, more generally it can be shown that an eigenmeasure can be designed by a delta function supported over any solution, as noted in other contexts in terms of the adjoint operator (Frobenius-Perron), \cite{bollt2013applied,bollt2000controlling,bollt2002manifold}, but in any case these are not functions.

Now re-examining this same issue leading to the eigenfunction from the first principle of referring to the definitions of spectral decomposition, Definition \ref{defn2} and Eq.~(\ref{spectraldecomp}), we demand eigen-like objects that are only allowable actual functions.  This is leads to the idea of ``approximate eigenfunctions" already included as part of spectral theory of operators in the field of functional analysis:

 \begin{definition}{\bf Approximate Eigenvalue.}\label{approxeig_sec}
Under the same assumptions as Definition \ref{defn2}, $b=e^\lambda$ is an  approximate eigenvalue  of a linear operator $K_t$ if there is no $f\neq 0$ such that $K_t f= b^t v$ (thus $b$ is not an element of the point spectrum) but it satisfies  that for any $\epsilon>0$, there exists a function $f_\epsilon:U\rightarrow {\mathbb C}, f_\epsilon \in {\cal F}$, such that
\begin{equation}
    \|K_t f_\epsilon - b^t f_\epsilon\|<\epsilon,
\end{equation}
  \end{definition}
 Furthermore the point is that $b$  must be an element of either the continuous spectrum $C_\sigma(K_t)$ or residual spectrum $R_\sigma(K_t)$, but not the point spectrum.  Therefore to show a nonempty $C_\sigma(K_t)\cup R_\sigma(K_t)$, we need only show existence of an approximate eigenfunction.  In brief, the reason this is sufficient is that an unbounded linear operator is therefore not a continuous operator \cite{boccara1990functional, yosida1994functional}.  
 
 To this end, recall that the action of the ``delta function", while not literally a function but rather a generalized function, it is nonetheless described weakly as the limit of the action of a sequence of real functions.  One such commonly used sequence of functions are the following specific set of indicator functions, $\delta_n:{\mathbb R}\rightarrow {\mathbb R}$, and notice the subindex $n\in {\mathbb Z}^+$,
 \begin{equation}\label{deltan}
     \delta_n(x)=\left( 
     \begin{array}{ccc}
     n & \mbox{ } & \frac{-1}{2n}<x<\frac{1}{2n} \\
     0 & \mbox{ } & \mbox{else}
     \end{array}\right)
 \end{equation}
 Likewise, a sequence indicator functions can be defined on a complex domain, or otherwise, many other smooth functions are also often used, notably radial basis functions, to build the action of the delta function as a sequence of actual functions.
 So with $\delta_n$ and mirroring Eq.~(\ref{delta}), consider a (candidate) approximate eigenfunction,
 \begin{equation}\label{approxeig}
      \phi_{\omega,n}(I,\theta)=e^{i\theta}\delta_n(I-\omega).
 \end{equation}
 Then, 
 \begin{eqnarray}
\|K_t \phi_{\omega,n}(I,\theta) - e^{i \omega t} \phi_{\omega,n}(I,\theta) \|_{L^2(M)} &=& 
\|\phi_{\omega,n}(I,\theta+It) - e^{i \omega t} \phi_{\omega,n}(I,\theta) \|_{L^2(M)}   \nonumber \\
&=& \|\phi_{\omega,n}(I,\theta+It) - e^{i \omega t} \phi_{\omega,n}(I,\theta) \|_{L^2(M)}
\nonumber \\ 
&=& \|  e^{i n\theta} (e^{iIt}   -   e^{i \omega t}) \delta_n (I-\omega)\|_{L^2(M)} \nonumber \\ 
&=& \|  e^{i n\theta} (e^{iIt}   -   e^{i \omega t}){\mathbf 1}_{(\frac{-1}{2n},\frac{1}{2n})}(I-\omega) \delta_n (I-\omega)\|_{L^2(M)} \nonumber \\ 
&\leq & \|  e^{i n\theta}\|_{L^2(M)} \| (e^{iIt}   -   e^{i \omega t}) {\mathbf 1}_{(\frac{-1}{2n},\frac{1}{2n})}(I-\omega)  \|_{L^2( M)}    \|  \delta_n (I-\omega)  \|_{L^2(M)} \nonumber \\
&= & \|  e^{i n\theta}\|_{L^2(M)} \| (e^{iIt}   -   e^{i \omega t})\|_{L^2( [\frac{-1}{2n},\frac{1}{2n}] \times S^1)}    \|  \delta_n (I-\omega)  \|_{L^2(M)} \nonumber \\
&=& \| (e^{iIt}   -   e^{i \omega t})\|_{L^2( [\frac{-1}{2n},\frac{1}{2n}] \times S^1)}     \nonumber \\ 
 &\approx & \| \mathcal{O} (I-\omega)\|_{L^2( [\frac{-1}{2n},\frac{1}{2n}] \times S^1)} \nonumber \\
 &=&\epsilon(n)
 \end{eqnarray}
 The third line follows substitution of the assumed form, Eq.~(\ref{approxeig}).
 The step to the fourth  line of the string of equalities  takes note that the chosen $\phi_{\omega,n}$ includes an indicator function, as $\delta_n$ in  Eq.~(\ref{deltan}) is zero outside the domain $ (\frac{-1}{2n},\frac{1}{2n})$.
 The seventh line simply notes that, $\|  e^{i n\theta}\|_{L^2(M)} =  \|  \delta_n (I-\omega)  \|_{L^2(M)}=1$.
 The second to last line follows a Taylor series of the exponential, and then finally the last line follows by defining, 
$ \epsilon(n)\sim \frac{c}{n},
$ for a constant $c$.
 Thus we conclude that $\delta_n$ are approximate eigenfunctions, and since approximate eigenfunctions exist, then $C_\sigma(K_t)\cup R_\sigma(K_t)$ is nonempty.  Furthermore this construction is general and leads to $C_\sigma(K_t)\cup R_\sigma(K_t)={\mathbb C}$, and from the material reviewed above $P_\sigma(K_t)=\emptyset$.
 
 We summarize: Contrasting the result in this subsection to the previous subsection, the same system the Eq.~(\ref{actionangleode}) in action-angle form, leads to essentially opposite outcomes depending on the domain chosen.   Either a simple system of all point spectrum and empty continuous union residual spectrum if a nonrecurrent subdomain is chosen, or otherwise an empty point spectrum but full continuous union residual spectrum.  For other reasons not discussed here we expect the continuous spectrum alone is full and this dovetails with much of the discussion regarding integration against an absolutely continuous measure in \cite{govindarajan2019approximation, korda2018data}.

\section{Appendix: Proof}\label{pflin}

Here we state a proof of Proposition \ref{algebraicprop}.  This is a result  which is well known results, \cite{budivsic2012applied}, but we include this proof in part because we have not seen this proof that appeals directly to the linear PDE, and we find it interesting.

\begin{proof}
Let $(\lambda_1,\phi_1)$ and $(\lambda_2,\phi_2)$ each be KEIGS. By Eq.~(\ref{eq:koopman definition}), let $b_1=e^{\lambda_1}$, and $b_2=e^{\lambda_2}$.  Each of these satisfies  Eq.~(\ref{kooppde}), $F\cdot \nabla \phi_{\lambda_1}=\lambda_1 \phi_{\lambda_1}$, and $F\cdot \nabla \phi_{\lambda_2}=\lambda_2 \phi_{\lambda_2}$.  Now we substitute $(\phi_{\lambda_1}(x))^{\alpha_1}(\phi_{\lambda_2}(x))^{\alpha_2})$ into the Koopman PDE,
\begin{eqnarray}
    F\cdot \nabla ((\phi_{\lambda_1}(x))^{\alpha_1}(\phi_{\lambda_2}(x))^{\alpha_2})) &=& 
    F\cdot(\alpha_1 (\phi_{\lambda_1})^{\alpha_1-1} \nabla \phi_{\lambda_1}\phi_{\lambda_2}^{\alpha_2}+\alpha_2(\phi_{\lambda_2})^{\alpha_2-1}\nabla \phi_{\lambda_2}\phi_{\lambda_1}^{\alpha_1}) \nonumber \\
    &=& F\cdot(\alpha_1 (\phi_{\lambda_1})^{\alpha_1-1} \phi_{\lambda_2}^{\alpha_2}\nabla \phi_{\lambda_1}+\alpha_2(\phi_{\lambda_2})^{\alpha_2-1}\phi_{\lambda_1}^{\alpha_1}\nabla \phi_{\lambda_2}) \nonumber  \\
    &=& \alpha_1 \phi_{\lambda_1}^{\alpha_1-1} \phi_{\lambda_2}^{\alpha_2}(F\cdot\nabla \phi_{\lambda_1})+\alpha_2\phi_{\lambda_2}^{\alpha_2-1}\phi_{\lambda_1}^{\alpha_1}(F\cdot\nabla \phi_{\lambda_2}) \nonumber \\
    &=& \alpha_1 \phi_{\lambda_1}^{\alpha_1-1} \phi_{\lambda_2}^{\alpha_2}(\lambda_1 \phi_{\lambda_1})+\alpha_2\phi_{\lambda_2}^{\alpha_2-1}\phi_{\lambda_1}^{\alpha_1}(\lambda_2\phi_{\lambda_2}) \nonumber \\
    &=&(\alpha_1 \lambda_1+\alpha_2 \lambda_2) \phi_{\lambda_1}^{\alpha_1}\phi_{\lambda_2}^{\alpha_2}.
\end{eqnarray}
Note also that,
\begin{equation}
    e^{\alpha_1 \lambda_1+\alpha_2 \lambda_2}=e^{\alpha_1 \lambda_1}e^{\alpha_2 \lambda_2}=(e^{\lambda_1})^\alpha_1 (e^{\lambda_2})^{\alpha_2})=(b_1)^{\alpha_1}(b_2)^{\alpha_2}.
\end{equation}
This relates the eigenvalues $\lambda_1, \lambda_2$ to the multiplying factors $b_1, b_2$.
\end{proof}
Note that since the theorem is stated, if $\phi_\lambda(x)$ exists in $U$, noting primarily that exponentiation by non-integers can potentially restrict domains.  For example, complex exponents force a consideration of an appropriate domain for a branch cut of the otherwise resulting Riemann surface.

{\color{black}
\section{Appendix: Illustration of Domain when stating Eigenfunctions}\label{fixapp}

The eigenfunctions derived by the PDE Eq.~(\ref{kooppde}) from Theorem \ref{thekoopmanpdetheorem} 
depend on the domain $U$ as emphasized in the theorem.  We illustrate this concept with a simple example, which appears in other context in many places including our own prior work, \cite{bollt2018matching} in the context of discussion of the domain of the KEIGS.  Consider the initial value problem in one dimension, 
\begin{equation}
\dot{x}=x^2, x(0)=x_0 \in {\mathbb R}/\{0\}.
\end{equation}
This is a classic example used to illustrate blow-up in finite time, since,
\begin{equation}
x(t)=\frac{1}{\frac{1}{x_0}-t},
\end{equation}
and $x(t)\rightarrow \infty$ if $x_0>0$ as $t\rightarrow \frac{1}{x_0}$, or $x(t)\rightarrow 0$ if $x_0<0$.
The Koopman eigenfunction equation Eq.~(\ref{kooppde}) is an ODE in this context due to the one-dimensional domain.
\begin{equation}\label{dme}
F(x) \cdot \nabla \phi_\lambda = x^2 \frac{d \phi_\lambda}{d x} = \lambda \phi_\lambda(x),
\end{equation}
which can be solved by integrating factors, 
\begin{equation}
\phi_\lambda(x)= e^{\frac{-\lambda}{x}}.
\end{equation}
This eigenfunction is not continuous on the full real line since there is an essential singularity at $x=0$.  However, it is smooth on any connected compact domain $U$, that is either $U\subset (-\infty,0)$ or $U\subset (0,\infty)$.  While we solved  Eq.~(\ref{dme}) directly by integrating factors, so not needing the method of characteristics, it nonetheless can be taken as descriptive of propagating initial data $h(x_0)$ from  $x_0\in \Lambda$,  in this setting that is just a singleton as this the domain is a line. Thus, $\phi_\lambda(x)$ can be smoothly defined for any such $U$.

}

\bibliographystyle{plain}
\bibliography{bibit}

\end{document}